\documentclass[a4paper, 10pt]{amsart}
\usepackage{amsmath,amsfonts,amssymb}
\usepackage{url}
\usepackage{hyperref}
\usepackage{comment}
\usepackage{bm} % for bold math symbols

\newtheorem{thm}{Theorem}[section]
\newtheorem{cor}[thm]{Corollary}
\newtheorem{lem}[thm]{Lemma}
\newtheorem{prop}[thm]{Proposition}
\newtheorem{defn}{Definition}%[section]
\theoremstyle{remark}
\newtheorem*{rmk}{Remark}%[section]
\newtheorem{example}{Example}%[section]
\newtheorem*{ex-cont}{Example}%[section]

\def\Z{{\mathbb Z}}
\def\Zhat{\hat{\Z}}
\def\Q{{\mathbb Q}}
\def\R{{\mathbb R}}
\def\A{{\mathbb A}}
\def\F{{\mathbb F}}

\def\CC{{\mathcal C}}
\def\E{{\mathcal E}}
\def\T{{\mathcal T}}
\newcommand\TT[3]{{\mathcal T}_{#1,#2,#3}}
\def\W{{\mathcal W}}
\def\Wodd{\W_{\text{odd}}} \def\Weven{\W_{\text{even}}}
\def\U{{\mathcal U}}
\def\B{{\mathcal B}}
\def\a{\bm{a}}
\def\x{\bm{x}}
\def\abar{\a(p^{6})}
\def\Ea{E_{\a}}
\DeclareMathOperator{\disc}{disc}

\DeclareMathOperator{\height}{ht}
\def\ordp{v}
\DeclareMathOperator{\GL}{GL}
\newcommand{\I}{\mathrm{I}}
\newcommand{\II}{\mathrm{II}}
\newcommand{\III}{\mathrm{III}}
\newcommand{\IV}{\mathrm{IV}}
\def\<#1>{\left<#1\right>}
\newcommand{\veq}[1]{=\kern-2pt#1}
\newcommand{\uk}{\mathbf{k}}
\newcommand{\Sage}{{\tt Sage}}

\newcommand{\lmfdbec}[3]{\href{http://www.lmfdb.org/EllipticCurve/Q/#1#2#3}{{\text{\rm#1#2#3}}}}

\begin{document}
\title[Densities for Weierstrass models of elliptic curves]{Local and global densities for Weierstrass models of elliptic curves}
\author{J. E. Cremona}
\address{Mathematics Institute, University of Warwick, Coventry CV4 7AL, UK}
\email{j.e.cremona@warwick.ac.uk}
\author{M. Sadek}
\address{Faculty of Engineering and Natural Sciences,
 Sabanc{\i} University,
  Tuzla, \.{I}stanbul, 34956 Turkey}
\email{mmsadek@sabanciuniv.edu}
\date{\today}
\dedicatory{In memory of John Tate, 1925--2019}

\begin{abstract}
We prove local results on the $p$-adic density of elliptic curves over
$\Q_p$ with different reduction types, together with global results on
densities of elliptic curves over $\Q$ with specified reduction types
at one or more (including infinitely many) primes.  These global
results include: the density of integral Weierstrass equations which
are minimal models of semistable elliptic curves over~$\Q$ (that is,
elliptic curves with square-free conductor) is
$1/\zeta(2)\approx60.79\%$, the same as the density of square-free
integers; the density of semistable elliptic curves over~$\Q$ is
$\zeta(10)/\zeta(2)\approx60.85\%$; the density of integral
Weierstrass equations which have square-free discriminant is
$\prod_p\left(1-\frac{2}{p^2}+\frac{1}{p^3}\right) \approx 42.89\%$,
which is the same (except for a different factor at the prime~$2$) as
the density of monic integral cubic polynomials with square-free
discriminant (and agrees with a 2013 result of Baier and Browning
for short Weierstrass equations); and the
density of elliptic curves over~$\Q$ with square-free minimal
discriminant is
$\zeta(10)\prod_p\left(1-\frac{2}{p^2}+\frac{1}{p^3}\right)\approx42.93\%$.

The local results derive from a detailed analysis of Tate's Algorithm,
while the global ones are obtained through the use of the Ekedahl
Sieve, as developed by Poonen, Stoll, and Bhargava.
\end{abstract}

\maketitle

%\tableofcontents

\section{Introduction}
\label{sec:intro}
In this paper we first study purely local results on the $p$-adic
density of elliptic curves over $\Q_p$ with different reduction types,
and then apply these, using a version of the Ekedahl Sieve, to
establish several global results on densities of elliptic curves over
$\Q$.

In the local setting, we use Tate's Algorithm \cite{tate1975algorithm}
to determine in Section~\ref{sec:localI} the
local density of Weierstrass equations having each possible reduction
type.  For example, the proportion of Weierstrass equations
over~$\Z_p$ which have good reduction (at~$p$) is $1-1/p$, those with
reduction of type~$\I_m$ (respectively~$\I_m^*$) have
density~$(p-1)^2/p^{m+2}$ (respectively~$(p-1)^2/p^{m+7}$), and the
density of elliptic curves over~$\Q_p$ which are semistable is
$(1-p^{-2})/(1-p^{-10})$.  See Propositions~\ref{prop:TAtable}
and~\ref{prop:typesIm} for details.  Here we distinguish between the
set of local integral Weierstrass equations with some property, and
the larger set of those which may not be minimal models but define
elliptic curves whose minimal model has the property.  For example,
the density of integral Weierstrass equations defining elliptic curves
with good reduction is $(1-p^{-1})/(1-p^{-10})$, which is greater than
the density $1-p^{-1}$ of equations which are themselves minimal
models of curves with good reduction, after allowing for non-minimal
models, as the local density of non-minimal Weierstrass equations is
$p^{-10}$.

We show that the local densities of minimal Weierstrass equations with
prime conductor and prime discriminant are, respectively, $(p-1)/p^2$
and $(p-1)^2/p^3$.

The local results mentioned so far all generalise immediately to any
$p$-adic field, replacing $p$ in each formula with the cardinality of
the residue field.

Further local results over~$\Q_p$ are obtained in
Section~\ref{sec:localII}, again by studying Tate's Algorithm in great
detail.  In Theorems~\ref{thm:fp-cp-count} and~\ref{thm:fp-all}, we
establish the densities of elliptic curves over~$\Q_p$ with each
possible conductor exponent and each possible Tamagawa number. (Note
that Tate's Algorithm in~\cite{tate1975algorithm} includes the
determination of both these quantities.) For example (see
Theorem~\ref{thm:fp-all}), among elliptic curves over~$\Q_3$ with
additive reduction the densities of the possible conductor
exponents~$f_3=2,3,4,5$ are in the ratio~$189:366:122:61$ or
approximately $25.6\%:49.6\%:16.5\%:8.3\%$.  Extending these results
to general extensions of~$\Q_p$ is not so straightforward, as the
analysis depends on the precise valuations of certain integers (such
as the coefficients of the discriminant of a long Weierstrass
equation).

In order to pass from local results to global statements, we make use
of a version of the Ekedahl Sieve from \cite{Ekedahl-sieve} as
developed by Poonen and Stoll in \cite{Poonen-Stoll} and further by
Bhargava in \cite{Bhargava-sieve}, by Bhargava, Shankar and Wang in
\cite{bhargava2016squarefree}, and elsewhere.  Provided that certain
conditions are met, it is often the case that global densities may be
expressed as a convergent infinite product (over all primes) of local
densities.  In order to be able to apply these methods with some
flexibility, we develop them systematically in
Section~\ref{sec:general-density}.

The global results, for elliptic curves over~$\Q$, follow
in~Section~\ref{sec:global}.  For a set~$S$ of Weierstrass equations
with integer coefficients~$\a=(a_1,a_2,a_3,a_4,a_6)\in\Z^5$, we define
the \emph{weighted density} of~$S$ to be
\begin{equation}\label{eqn:def-density-intro}
\rho^{\uk}(S) = \lim_{X\to\infty}\frac{\#\{\a\in S\mid |a_i|\le X^i\}}{\#\{\a\in \Z^5\mid |a_i|\le X^i\}},
\end{equation}
when this limit exists, where $\uk=(1,2,3,4,6)$.  (More general
weighted densities will be defined in subsection~\ref{subsec:globalI}
below.)  An alternative way of expressing density results is to define
the \emph{height} of a Weierstrass equation with integer
coefficients~$\a$ to be
\[
\height(\a) = \max_i|a_i|^{1/i},
\]
and then order such equations by height; then we may say that when
integral Weierstrass equations are ordered by height, the proportion
which lie in the set~$S$ is $\rho^{\uk}(S)$, whose definition may now be
written as
\[
\rho^{\uk}(S) = \lim_{X\to\infty}\frac{\#\{\a\in S\mid \height(\a)\le X\}}{\#\{\a\in \Z^5\mid \height(\a)\le X\}}.
\]
Each of our results will have two versions, depending on whether we
restrict to Weierstrass equations which are globally minimal, or
include all equations.  (Those with zero discriminant, which define
singular curves, may always be ignored as they form a set of measure zero.)

In general, the global density exists and equals the product of the
corresponding local densities, provided that the local condition
specified at all but finitely many primes is to have good or
multiplicative reduction.  We state here a summary of the results from
Section~\ref{sec:global}, which allow more flexibility in specifying
local conditions at any finite set of primes.

\begin{thm}\label{thm:main-global}
When ordered by height, the proportion of integral Weierstrass
equations with each of the following properties is as given:
\begin{itemize}
\item globally minimal: $1/\zeta(10)=93555/\pi^{10} \approx 99.9\%$;
\item minimal models of semistable elliptic curves:
  $1/\zeta(2)=6/\pi^2 \approx 60.8\%$;
\item minimal models of semistable elliptic curves with good reduction
  at all the primes in the finite set~$S$:  $\zeta(2)^{-1}\prod_{p\in
    S}\frac{p}{p+1}$;
\item minimal models of elliptic curves with square-free
  discriminant:
  $$\prod_{p}\left(1-\frac{2}{p^2}+\frac{1}{p^3}\right)
  \approx 42.9\%.$$
\end{itemize}
In each case, the proportion of integral Weierstrass equations which
are not necessarily minimal models of elliptic curves with the stated
property may be obtained by multiplying by~$\zeta(10)\approx 1.001$.
\end{thm}

It would be interesting to extend the global results here to number
fields other than~$\Q$, which would entail several additional
challenges.

\subsection{Related work}
Our result for the density of integral Weierstrass equations which
have square-free discriminant is---apart from a different local factor
at~$2$---the same as the density of monic integral cubic polynomials
with square-free discriminant: see the 2016 paper
\cite{bhargava2016squarefree} of Bhargava \textit{et al.}, and also
Theorem 6.8 in the 2007 paper~\cite{Ash2007} by~Ash, Brakenhoff, and Zarrabi.
We note that this is also in agreement with a result of Baier and
Browning in their 2013 paper~\cite{Baier-Browning} (see also Baier's
2016 paper~\cite{Baier-squarefree}) for short Weierstrass equations
$Y^2=X^3+AX+B$ with squarefree discriminant, established using quite
different methods.

In their famous 1990 paper~\cite{brumer1990behavior}, Brumer and
McGuinness give heuristics for the number of elliptic curves whose
minimal discriminant is less than~$X$, separating the cases of
positive and negative discriminant. In each case the number is
conjectured to be a constant multiple of~$X^{5/6}$ with a constant
which is the value of an elliptic integral divided (in each case)
by~$\zeta(10)$, the latter to allow for non-minimal discriminants.
This was revisited by Watkins in 2008 in~\cite{watkins2008some}, who
re-derives the same heuristic estimate, and also discusses the
factor~$\zeta(10)$.  Watkins also gives applications to the
distribution of curves by conductor instead of discriminant, and also
to the distribution of odd and even analytic ranks.

Some similar results, including local conditions, are given in the
2001 paper~\cite{Wong-densities} of Wong, who defines the height of an
elliptic curve over~$\Q$ to be
$\height_c(E)=\max\{|c_4(E)|^{1/4},|c_6(E)|^{1/6}\}$, where
$c_4(E),c_6(E)$ are the invariants of a minimal model for~$E$.  This
is comparable with our height: in one direction, standard formulae for
$c_4,c_6$ imply that, for $E$ defined by a minimal Weierstrass
equation with coefficients~$\a$, we have $\height_c(E)\ll\height(\a)$.
In the other direction, given a pair~$(c_4,c_6)$ which satisfy Kraus's
conditions from~\cite{Kraus-c4c6}, using the formulas in the first
author's book~\cite[p.~61]{JCbook2} to recover Weierstrass
coefficients~$\a$ from these, one obtains $\height(\a) \ll
\max\{|c_4(E)|^{1/4},|c_6(E)|^{1/6}\}$.  In Theorem~1
of~\cite{Wong-densities}, Wong gives asymptotic expansions of the
number of curves of height up to~$X$ together with the number which
are semistable, and the number which are semistable and have good
reduction at both~$2$ and~$3$.  In each case, the leading coefficient
gives the value of the density in our sense.  To compare these with
our results, we first need to take into account the density of
$(c_4,c_6)$ pairs which satisfy Kraus's conditions, which may easily
be seen to be $2^{-7}3^{-3}=1/3456$, and also the minimality condition
which leads to a factor of
$1/\zeta(10)=9355/\pi^{10}=3^5\cdot5\cdot7\cdot11/\pi^{10}$ as in our
theorem above.  The density given in~\cite{Wong-densities} is a
rational multiple of~$1/\pi^{10}$, but with a different rational
factor.  We would also expect, from our theorem above, that the
density for semistable curves should be multiplied by
$1/\zeta(2)=6/\pi^2$, and that if in addition we impose the condition
of having good reduction at~$2$ and~$3$, the density should be
multiplied by~$(2/3)(3/4)=1/2$, rather than $7/9$ as
in~\cite{Wong-densities}.  These discrepancies lead to Wong's
statement that the proportion of semistable curves is~$17.9\%$,
compared with our value of~$60.85\%$.  We should emphasize that the
majority of Wong's results in~\cite{Wong-densities} do not depend on
precise values of any densities, only that they exist and are
positive.

Over general number fields, not all elliptic curves have global
minimal Weierstrass equations when the class group is non-trivial.  In
her 2004~paper~\cite{bekyel2004density}, Bekyel determined the density
of elliptic curves defined over any number field~$K$ which have global
minimal models to be $\zeta_K(\CC_0,10)/\zeta_K(10)$, where
$\zeta_K(s)$ is the Dedekind zeta function of~$K$ and
$\zeta_K(\CC_0,s)$ is the partial zeta function associated to the
trivial ideal class.  Of course this equals~$1$ when the class group
is trivial. Note that once again the factor of~$\zeta_K(10)$ appears.

Earlier work of Papadopoulos in~\cite{papadopoulos} uses a close
analysis of Tate's Algorithm similar to our approach in
Sections~\ref{sec:localI} and~\ref{sec:localII}, working over a
general local field, but not including the quantification of the local
densities which we require.

This paper grew out of independent work of each of the authors:
unpublished notes on purely local densities (at arbitrary primes) by
Cremona, and a 2017~preprint \cite{sadek2017counting} on global
densities (excluding conditions at the primes~$p=2$ and~$p=3$) by
Sadek.  After the first version of the current paper appeared online,
we noticed a new preprint~\cite{ChoJeong2017counting} by Cho and
Jeong, whose subject matter has some overlap with the current paper,
but with several differences: conditions at the primes~$2$ and~$3$ are
excluded in~\cite{ChoJeong2017counting}, and only conditions at
finitely many primes are considered, through the use of short
Weierstrass equations.  On the other hand, they consider additional
local conditions we do not, including the condition of having a fixed
trace of Frobenius~$a_p$ at a prime~$p$ of good reduction, and their
paper also contains applications to the distribution of analytic
ranks.

\subsection*{Acknowledgements:} The authors would like to thank Manjul Bhargava,
Adam Harper, Bjorn Poonen, and Michael Stoll for useful suggestions,
and the anonymous referee, whose comments led to an improved
exposition, supported by computer code now available
at~\cite{sagecode}, of the details in Section~\ref{sec:localII}.
The first author was supported by EPSRC Programme Grant EP/K034383/1
\textit{LMF: L-Functions and Modular Forms}, the Horizon 2020 European
Research Infrastructures project \textit{OpenDreamKit} (\#676541), and
the Heilbronn Institute for Mathematical Research.

%%%%%%%%%%%%%%%%%%%%%%%%%%%%%%%%%%%%%%%%%%%%%%%%%%%%%%%%%%%%%%%%%%%%%%

\section{Local densities I}\label{sec:localI}

\subsection{Weierstrass equations and coordinate transformations}

For any integral domain $R$ denote by
\[
\W(R) = R^5 = \{\a=(a_1,a_2,a_3,a_4,a_6)|a_i\in R\}
\]
the set of all $5$-tuples of coefficients in~$R$ of
plane cubic curves~$\Ea$ in long Weierstrass form over $R$:
\[
\Ea:\ Y^2+a_1XY+a_3Y=X^3+a_2X^2+a_4X+a_6.
\]
Denote by $\Delta(\a)$ the discriminant of~$\Ea$; when $\Delta(\a)$ is
non-zero, $\Ea$ is a model for an elliptic curve defined over the
fraction field of~$R$; otherwise, we say that $\a$ is \emph{singular}.
Below we will also refer to the standard associated quantities $b_2$,
$b_4$, $b_6$, $b_8$, $c_4$ and~$c_6$; together with~$\Delta$ these may
all be viewed as elements of~$\Z[a_1,a_2,a_3,a_4,a_6]$.

%% $b_2 = a_1^2+4a_2$, $b_4 = a_1a_3+2a_4$, $b_6 = a_3^2+4a_6$,
%% $b_8=a_1^2a_6-a_1a_3a_4+4a_2a_6+a_2a_3^2-a_4^2$

The \emph{translation group} $\T(R)=\{\tau(r,s,t)\mid r,s,t\in R\}$
acts on $\W(R)$ in the standard way, with $\tau(r,s,t)$ induced by the
coordinate substitutions $(X,Y)\mapsto(X+r,Y+sX+t)$; we call elements
of~$\T(R)$ \emph{translations}.

In the case $R=\Z_p$, we make further definitions of certain subsets
of $\W(\Z_p)$ and subgroups of $\T(\Z_p)$.  We denote by~$\ordp$ the
normalised $p$-adic valuation.

Given non-negative integers~$v_i$ for $i=1,2,3,4,6$, define
\[
 \W(v_1,v_2,v_3,v_4,v_6) = \{\a\in\W(\Z_p)\mid \ordp(a_i)\ge
 v_i\ \text{for~$i=1,2,3,4,6$}\}.
 \]
To specify further that $\ordp(a_i)=v_i$ exactly, we indicate this by
writing ``$\veq{v_i}$'': for example, $\W(1,1,1,1,\veq1)$.  Below we
will also need notation for subsets of these satisfying an additional
condition, for example $\W(1,1,1,2,2\mid\ordp(b_2)=2)$ and
$\W(1,1,1,2,2\mid\ordp(\Delta)=6)$, whose meaning should be clear.

For $e,f,g\ge0$ we define
\[
\TT{e}{f}{g} = \{\tau(r,s,t)\in\T(\Z_p) : p^e\mid r,\ p^f\mid s,\ p^g\mid t\},
\]
which is a subgroup of $\T(\Z_p)$ provided $e+f\ge g$, of index
$p^{e+f+g}$.

\subsection{Local densities and Tate's Algorithm}

For each non-singular $\a\in\W(\Z_p)$, the equation $\Ea$ defines an
elliptic curve over $\Q_p$.  With the usual $p$-adic measure~$\mu$ on
$\Z_p$ such that $\mu(\Z_p)=1$, we have $\mu(\W(\Z_p))=1$, and for any
measurable subset $S\subseteq\W(\Z_p)$ we refer to $\mu(S)$ as the
\emph{density} (or $p$-adic density) of the associated set of
equations~$\Ea$, and also think of $\mu(S)$ as the probability that a
random Weierstrass equation lies in $S$.  Note that the subset of
singular~$\a$ has measure zero, and may be tacitly ignored.

For example,
\begin{equation}
   \mu(\W(v_1,v_2,v_3,v_4,v_6))=1/p^{v_1+v_2+v_3+v_4+v_6}, \label{eqn:W-density1}
\end{equation}
while if any of the $v_i$ is replaced by $\veq{v_i}$, then the measure
should be multiplied by $(1-1/p)$; so for $i=6$,
\begin{equation}
   \mu(\W(v_1,v_2,v_3,v_4,\veq{v_6}))=(p-1)/p^{v_1+v_2+v_3+v_4+v_6+1},  \label{eqn:W-density2}
\end{equation}
and similarly for $i=1,2,3,4$.

We write $\T=\T(\Z_p)$ for the rest of this section.  The action of
$\T$ on $\W(\Z_p)$ is measure-preserving and also leaves the
discriminant, and $c_4$ and $c_6$, invariant.  Translations induce
isomorphisms of elliptic curves (when $\Delta\not=0$).

Given some property or \emph{type} $T$ of isomorphism classes of
elliptic curves over $\Q_p$, we associate a subset
$\W_T(\Z_p)\subseteq\W(\Z_p)$:
\[ \W_T(\Z_p) = \{\a\in\W(\Z_p) \mid \Ea\ \text{is smooth and has
  type}\ T\},
\]
and define the density of curves with property~$T$ as the $p$-adic
measure of this set.
\begin{defn}\label{def:rhoT} The local density $\rho_T(p)$ of
elliptic curves over $\Q_p$ with type $T$ is the $p$-adic measure
$\mu(\W_T(\Z_p))$ of the associated subset
$\W_T(\Z_p)\subseteq\W(\Z_p)$:
\[
\rho_T(p) = \mu(\W_T(\Z_p)).
\]
\end{defn}
In this section $p$ is fixed and we abbreviate: $\W_T=\W_T(\Z_p)$ and
$\rho_T=\rho_T(p)$.

The types of interest to us are the following Kodaira types of
reduction of elliptic curves over $\Q_p$:
\begin{itemize}
  \item $\I_0$ (good reduction);
  \item $\I_{\ge1}$ (bad multiplicative reduction, of type $\I_m$ for some
  $m\ge1$);
  \item bad additive reduction; with subtypes
     $\II$, $\III$, $\IV$, $\II^*$, $\III^*$, $\IV^*$, $\I_0^*$,
    $\I_{\ge1}^*$, the latter meaning type $\I_m^*$ for some $m\ge1$.
\end{itemize}
We call these types \emph{finite}, since, as we will see below (see
Proposition~\ref{prop:mod6}), we only need know the coefficients~$\a$
to finite $p$-adic precision in order to determine whether the
curve~$E_{\a}$ has each of these reduction types, provided that
$E_{\a}$ is a minimal model.  Moreover, the condition that $E_{\a}$ is
minimal also only depends on~$\a$ to finite precision (modulo~$p^6$
suffices).  Note that $\I_{\ge1}$ and $\I_{\ge1}^*$, the unions of
types $\I_m$ and $\I_m^*$ for all~$m\ge1$ respectively, are finite in
this sense.  However, while for each fixed~$m$ it is true that finite
$p$-adic precision suffices to detect the individual types~$\I_m$
and~$\I_m^*$, this precision depends on~$m$. For this reason we do not
regard these types as finite, and for some of our results it will be
necessary to consider these together, rather than individually.

Note that while each model $\Ea$ with $\Delta(\a)\not=0$ defines an
elliptic curve over $\Q_p$ whose type is well-defined, the set of
$\a\in\W_T(\Z_p)$ for which $\Ea$ is itself a minimal model is a
strictly smaller subset, with a smaller density, since scaling
(replacing each $a_i$ by $p^{ni}a_i$ for some $n\ge1$) does not change
the isomorphism class of $\Ea$.  We will relate these two densities.

Define
\[
\W_M = \{\a\in\W(\Z_p)\mid \Ea\ \text{is a minimal model}\},
\]
call $\a\in \W_M$ \emph{minimal}, and set $\W_N$ to be the
complement~$\W(\Z_p)\setminus \W_M$.  This complement contains the set
$\W(1,2,3,4,6)$ of all ``trivially non-minimal'' $\a$, satisfying
$p^i\mid a_i$ for $i=1,2,3,4,6$, which has measure $p^{-16}$.  It is
clear that the action of $\T$ preserves minimality, so $\T$ maps
both $\W_M$ and~$\W_N$ to themselves.
\begin{prop}\label{prop:trivmin}
\
  \begin{enumerate}
\item The subgroup of $\T$ preserving $\W(1,2,3,4,6)$ is $\TT213$.
\item Each orbit of $\T$ on $\W_N$ contains an  element of $\W(1,2,3,4,6)$.
\item $\mu(\W_M)=1-p^{-10}$.
\end{enumerate}
\end{prop}
\begin{proof}
(1) follows from the standard formulas linking the coefficients $\a$
  to the transformed coefficients~$\a'$ after translation by
  $\tau(r,s,t)\in\T$; in case $p\ge5$ this is almost trivial, and it
  is straightforward to check for $p=2$ and $p=3$.  See the proof of
  Theorem~\ref{thm:fp-cp-count} for details.

(2) follows directly from Tate's algorithm \cite{tate1975algorithm},
in which, given any non-minimal $\a$, one constructs a sequence of
translations taking $\a$ to some $\a'\in\W(1,2,3,4,6)$.

(3): from (2), since $\TT213$ has index~$p^6$ in $\T$, it
follows that $\W_N$ is partitioned into $p^6$ disjoint subsets, each a
translation of~$\W(1,2,3,4,6)$ by an element of one coset of
$\TT213$.  Since~$\mu(\W(1,2,3,4,6))=1/p^{16}$, it follows that
$\mu(\W_N)=p^6/{p^{16}}=1/p^{10}$ and hence $\mu(\W_M)=1-1/p^{10}$.
\end{proof}

For each type $T$, we set $\W_T^M = \W_T\cap \W_M$, the set of
$\a\in\W(\Z_p)$ for which $\Ea$ is a minimal model of an elliptic curve
of type~$T$, and make the following definition:
\begin{defn}\label{def:rhoTM}
The local density $\rho_T^M=\rho_T^M(p)$ of minimal Weierstrass
equations defining elliptic curves over $\Q_p$ of type $T$ is the
$p$-adic measure of $\W_T^M$:
\[
  \rho_T^M = \mu(\W_T^M) = \mu(\W_T\cap \W_M).
\]
\end{defn}

Although the properties we consider are invariants of elliptic curves
up to isomorphism over $\Q_p$, and not properties of specific models
or equations, we can still determine local densities by studying
Weierstrass models, by relating $\rho_T$ and $\rho_T^M$.  For example,
the model $\Ea$ will have bad reduction modulo~$p$ when
$\Delta(\a)\equiv0\pmod{p}$, but the curve over $\Q_p$ which this
model defines may still have good reduction if the model is
non-minimal.

Just as all non-minimal~$\a$ can be translated into the
set~$\W(1,2,3,4,6)$, which is defined by simple valuation conditions
on the coefficients, Tate's algorithm implies that, for each type~$T$,
there is a ``base set''~$\B_T$ also defined by valuation conditions,
such that
\[
\text{$\a$ is minimal and of type $T$}
\quad\iff\quad
\text{$\a$ has a translate in $\B_T$.}
\]
In the following proposition and table, we define such a
set~$\B_T\subseteq\W(\Z_p)$ for each finite type~$T$, and give its
measure and the subgroup~$\T_T\subseteq\T$ which stabilises it.  For
example, in the first line of the table for $T=\I_0$ (good reduction),
we have $\B_T=\W(0,0,0,0,0\mid\ordp(\Delta)=0)$, since the only
condition required for good reduction apart from integrality (all
coefficients have valuation~${}\ge0$) is that the discriminant has
valuation zero.  This condition is invariant under all translations,
so $\T_{\I_0}=\TT000=\T$.

\begin{prop}\label{prop:TAtable}
For each minimal~$\a\in\W(\Z_p)$ there exists $\tau\in\T$ such that
$\tau(\a)\in\B_T$ for exactly one of the base sets $\B_T$ in the
following table.  The table also shows the measure $\mu(\B_T)$, the
stabiliser $\T_T$ and its index, and the
measure~$\rho_T^M=\mu(\W_T^M)$.  The last row refers to the set of
non-minimal~$\a$, which has density~$1/p^{10}$, with base set the set of trivially
non-minimal~$\a$.  The discriminant of
the cubic $x^3+a_2x^2+a_4x+a_6$ is denoted\footnote{For $p\not=2$ we
  have $v(\tilde\Delta)=v(\Delta)$ for $\a\in\W(1,1,2,2,3)$, but this
  is not the case when $p=2$.} $\tilde\Delta$.
\[
\begin{tabular}{clcccc}
  \hline
  $T$ & $\B_T$ & $\mu(\B_T)$ & $\T_T$ & $[\T:\T_T]$ & $\rho_T^M$ \\
  \hline
  $\I_{0}$ & $\W(0,0,0,0,0\mid\ordp(\Delta)=0)$ & $(p-1)/p$ & $\TT000$ & $1$ & $(p-1)/p$ \\

  $\I_{\ge1}$ & $\W(0,0,1,1,1\mid\ordp(b_2)=0)$ & $(p-1)/p^4$ & $\TT101$
  & $p^2$ & $(p-1)/p^2$ \\

  $\II$ & $\W(1,1,1,1,\veq1)$ & $(p-1)/p^6$ & $\TT111$
  & $p^3$ & $(p-1)/p^3$ \\

  $\III$ & $\W(1,1,1,\veq1,2)$ & $(p-1)/p^7$ & $\TT111$
  & $p^3$ & $(p-1)/p^4$ \\

  $\IV$ & $\W(1,1,1,2,2\mid\ordp(b_6)=2)$ & $(p-1)/p^8$ & $\TT111$
  & $p^3$ & $(p-1)/p^5$ \\

  $\I_0^*$ & $\W(1,1,2,2,3\mid\ordp(\tilde\Delta)=6)$ & $(p-1)/p^{10}$ & $\TT112$
  & $p^4$ & $(p-1)/p^6$ \\

  $\I_{\ge1}^*$ & $\W(1,\veq1,2,3,4)$ & $(p-1)/p^{12}$ & $\TT212$
  & $p^5$ & $(p-1)/p^7$ \\

  $\IV^*$ & $\W(1,2,2,3,4\mid\ordp(b_6)=4)$ & $(p-1)/p^{13}$ & $\TT212$
  & $p^5$ & $(p-1)/p^8$ \\

  $\III^*$ & $\W(1,2,3,\veq3,5)$ & $(p-1)/p^{15}$ & $\TT213$
  & $p^6$ & $(p-1)/p^9$ \\

  $\II^*$ & $\W(1,2,3,4,\veq5)$ & $(p-1)/p^{16}$ & $\TT213$
  & $p^6$ & $(p-1)/p^{10}$ \\

  & $\W(1,2,3,4,6)$ & $1/p^{16}$ & $\TT213$
  & $p^6$ &  \\
  \hline
\end{tabular}
\]
\end{prop}
\begin{proof}
The conditions defining each basic set $\B_T$ in the table are
equivalent to the exit conditions in Tate's algorithm.  The density of
$\B_T$ is given by~\eqref{eqn:W-density1} or~\eqref{eqn:W-density2}
when there is no extra condition (such as $\ordp(b_6)=2$ for
type~$\IV$); the extra condition always has the effect of multiplying
the density by~$1-1/p$.  For type~$\I_0$ this
is Proposition~\ref{prop:good-density}, while for types~$\I_{\ge1}$, $\IV$,
and $\IV^*$ see \eqref{eqn:B_I1}, \eqref{eqn:B_IV},
and~\eqref{eqn:B_IVs} in Section~\ref{sec:localII} respectively.

The last column is the product of the index~$[\T:\T_T]$ and the
measure of $\B_T$, since the subset of~$\a$ of type $T$ is the
disjoint union of $[\T:\T_T]$ translates of $\B_T$.

The side conditions for types $I_{\ge1}$, $\IV$ and $\IV^*$ ensure
that a certain quadratic has distinct roots modulo~$p$, while that for
$I_0^*$ ensures that a certain cubic has distinct roots modulo~$p$.
In the algorithm, if the exit condition for types $\I_0$,
$\I_{\ge1}$, $\IV$, $\I_0^*$ and $\I_{\ge1}^*$ fails, a translation is
required before continuing, and hence the stabiliser becomes smaller,
by index~$p$ except in the first step when the index is~$p^2$.

Tate's Algorithm itself takes an arbitrary $\a\in\W(\Z_p)$ and applies
to it a sequence of translations, each well-defined up to an element
in the next stabiliser, until it has been transformed into one of the
base sets~$\B_T$, at which point one concludes that the reduction type
is~$T$, or that the equation was not minimal.

Additional detail will be given in the proof of
Theorem~\ref{thm:fp-cp-count} below.
\end{proof}

The next proposition implies that for each of the finite\footnote{Recall
  that we do not consider the individual types $\I_m$ and $\I_m^*$ as
  finite.}  types~$T$, the condition that $\a\in \W_T^M$ only depends
on the class of $(a_i\pmod{p^6})$ in $(\Z_p/p^6\Z_p)^5$.  We denote
this product by $\W(\Z_p/p^6)$ and the image of~$\a$ in $\W(\Z_p/p^6)$
by $\abar$; there are $p^{30}$ classes in $\W(\Z_p/p^6)$, each of
measure $1/p^{30}$.  Similarly for $\W_N$.

\begin{prop}\label{prop:mod6}
Let $\a,\a'\in\W(\Z_p)$ be such that $\abar=\a'(p^{6})$. Then
\[
\a\in \W_M \iff \a' \in \W_M,
\]
and for each finite type~$T$ we have
\[
\a\in \W_T^M\iff \a' \in \W_T^M.
\]
\end{prop}
\begin{proof}
This again follows from Tate's Algorithm. At each step the exit
criterion is a test for membership of one of the basis sets $\B_T$,
which only depends on $\abar$.  Also, whenever a coordinate
transformation $\tau(r,s,t)$ is required, in each case it is taken
from the finite set of cosets of one of the subgroups $\TT{e}{f}{g}$.
It is clear that the action of $\T$ is well-defined on
$\W(\Z_p/p^6)$, in the sense that for each $\tau\in\T$,
$\abar=\a'(p^{6})$ implies $\tau(\a)(p^6)=\tau(\a')(p^{6})$.

It follows that the outcome of the algorithm (up to the point of
determining that the initial~$\a$ was non-minimal, and excluding the
exact index~$m$ for types $\I_m$ and $\I_m^*$) also only depends on
the initial value of $\abar$.
\end{proof}

\begin{cor}\label{cor:NT}
For each finite type $T$,
\[
\rho_T^M = N(T)/p^{10}
\]
where $N(T)=p^k-p^{k-1}$ for some integer~$k$ with $1\le k\le10$,
depending on the type $T$, such that
\[
   \#\{\a\in \W_T^M\mid 0\le a_i<p^6\ \text{for $i=1,2,3,4,6$}\} = p^{20}N(T).
\]
\end{cor}
\begin{proof}
This follows immediately from the table above.
\end{proof}

The precise index $m$ for types $\I_m$ and $\I_m^*$ when $m\ge1$
depends on the discriminant valuation which can be arbitrarily large,
so no fixed $p$-adic precision will suffice to determine this value in
all cases.  However, for later reference we can determine the
densities of these types for each~$m$:
\begin{prop}\label{prop:typesIm}
For each $m\ge1$ we have $\rho_{I_m}^M=(p-1)^2/p^{m+2}$ and
$\rho_{I_m^*}^M=(p-1)^2/p^{m+7}$.
\end{prop}
\begin{proof}
Consideration of Tate curves shows that $\rho_{I_m}^M =
p\cdot\rho_{I_{m+1}}^M$.  Explicitly,
in~\cite[\S2.2]{JC-ComponentGroups}, the first author proved that when
$p^m\mid\Delta$, there is a translation of the form~$\tau(r,0,t)$ to a
Weierstrass model such that $p^m$ divides all
of~$a_3,a_4,a_6,b_4,b_6$, and~$b_8$.  For such a model we have
$\ordp(\Delta)=m \iff \ordp(b_8)=m \iff \ordp(a_6)=m$.  Hence the
relative density of models of type~$\I_{\ge m+1}$ within those of
type~$\I_{\ge m}$ is~$1/p$.  Since
$\sum_{m\ge1}\rho_{I_m}^M=\rho_{I_{\ge1}}^M=(p-1)/p^2$ (see
Proposition~\ref{prop:TAtable}), the first result follows.

For the second result, a careful analysis of Tate's algorithm (see the
proof of Theorem~\ref{thm:fp-cp-count} below) again shows that the
density is reduced by a factor of~$p$ when $m$ increases by~$1$, since
the criterion for increasing~$m$ is that a certain monic quadratic has
a repeated root modulo~$p$, which has probability $1/p$.
\end{proof}

The preceding proof shows that
\[
\B_{\I_m} = \W(0, 0, m, m, \veq\,{m} \mid \ordp(b_2)=0),
\]
with measure~$(p-1)^2/p^{3m+2}$ and stabiliser of index~$p^{2m}$.
In Section~\ref{sec:localII} we will show that
\[
   \B_{\I_m^*} = \begin{cases}
\W(1,\veq1,k+1,k+2,2k+2\mid\ordp(b_6)=2k+2) & \text{if $m=2k-1$;}\\
\W(1,\veq1,k+2,k+2,2k+3\mid\ordp(b_8)=2k+4) & \text{if $m=2k$;}
   \end{cases}
\]
with measure $(p-1)^2/p^{2m+11}$ and stabiliser of index~$p^{m+4}$.

Hence we have an explicit upper bound on the $p$-adic precision to
which we must know $\a\in\W(\Z_p)$ in order to determine the type of
$\Ea$, provided that $\a$ is minimal, for all finite types.  This is
false without the minimality condition---that is, we cannot replace
the subsets $\W_T^M$ by $\W_T$ in Proposition~\ref{prop:mod6}---since
scaling (replacing each $a_i$ by $p^{ni}a_i$ for some $n\ge1$) does
not change the isomorphism class of $\Ea$.  Later we will consider
$\a$ to higher $p$-adic precision in order to handle non-minimal
models.  On the other hand, for most finite types, lower $p$-adic
precision than $\abar$ is required: for example, to distinguish
between good reduction, multiplicative reduction and additive
reduction of a minimal model only requires knowledge of $\a\pmod{p}$.
However the individual finite types of additive reduction require
successively higher precision, as does the condition of minimality
itself, and to treat all finite types uniformly it is more convenient
to work modulo~$p^6$.

The two densities $\rho_T$ and $\rho_T^M$ are related as follows.
\begin{prop}\label{prop:rhoTM}
For each finite type $T$,
\[
    \rho_T = \frac{p^{10}}{p^{10}-1}\rho_T^M.
\]
\end{prop}
\begin{proof}
Recall from Proposition~\ref{prop:trivmin} that the set
$\W(1,2,3,4,6)$ of trivially non-minimal~$\a$ has measure~$p^{-16}$,
and the set $\W_N$ of all non-minimal~$\a$ is the union of $p^{6}$
translates of~$\W(1,2,3,4,6)$ under a set of translations $\tau$ which
are coset representatives for $\TT213$ in $\T$.

The scaling map
$(a_1,a_2,a_3,a_4,a_6)\mapsto(pa_1,p^2a_2,p^3a_3,p^4a_4,p^6a_6)$ is a
bijection from $\W_T$ to $\W_T\cap \W(1,2,3,4,6)$, so $\mu(\W_T\cap
\W(1,2,3,4,6))=p^{-16}\mu(\W_T)$.  Hence
\[
\mu(\W_T\cap \W_N) = p^{6}\mu(\W_T\cap \W(1,2,3,4,6)) = p^{-10}\mu(\W_T),
\]
so $\mu(\W_T\cap \W_M) = (1-p^{-10})\mu(\W_T)$ and hence $\rho_T^M =
(1-p^{-10})\rho_T$.
\end{proof}

Writing this relation as $\rho_T=\rho_T^M\sum_{k=0}^{\infty}p^{-10k}$,
we now give an interpretation of each term of the series in terms of
the ``level of non-minimality'' for~$\a\in\W(\Z_p)$, which we now
define.
\begin{defn}\label{def:lambda}
Let $\a\in\W(\Z_p)$ with $\Delta(\a)\not=0$.  The \emph{level}
$\lambda(\a)$ of~$\a$ is defined by
\[
 \lambda(\a) = \frac{1}{12}\left(\ordp(\Delta(\a))-\ordp(\Delta_{\min}(\Ea))\right),
\]
where $\Delta_{\min}(\Ea)$ is the discriminant of a minimal model for
the elliptic curve $\Ea$.
\end{defn}
With this definition, $\a$ is minimal if and only if $\lambda(\a)=0$,
and $\W(\Z_p)$ is the disjoint union of ``level sets''
$\W_k=\{\a\in\W(\Z_p)\mid\lambda(\a)=k\}$, together with the set of
singular~$\a$.  Let $\W_{T,k}=\W_T\cap\W_k$.

\begin{prop}\label{prop:muWk}
For each $k\ge0$,
\[
  \mu(\W_k) = (1-p^{-10})/p^{10k}
\]
and
\[
  \mu(\W_{T,k}) = \rho_T^M/p^{10k}.
\]
\end{prop}
\begin{proof}
For $k=0$ the first statement follows from the table and the second is
by definition, using $\W_T^M=\W_{T,0}$ and the definition of
$\rho_T^M$.  Proceeding by induction, scaling by $p$ maps $\W_k$ to
$\W_{k+1}\cap\W(1,2,3,4,6)$ whose measure is $\mu(\W_{k+1})/p^6$.
Hence $\mu(\W_{k})/p^{16} = \mu(\W_{k+1})/p^{6}$, so $\mu(\W_{k+1}) =
\mu(\W_k)/p^{10}$.  Similarly when we restrict to any fixed finite
type~$T$, we obtain $\mu(\W_{T,k+1}) = \mu(\W_{T,_k})/p^{10}$.
\end{proof}

This proof implies the following generalisation of the statements
above that minimality of~$\a$, and the type of $\Ea$ when minimal, only
depend on $\a\pmod{p^6}$.
\begin{cor}\label{cor:NTk}
  Let $k\ge0$.
\begin{enumerate}
\item
  The class of $\a\pmod{p^{6(k+1)}}$ determines $\lambda(\a)$ exactly
  if $\lambda(\a)\le k$.
\item
  When $\lambda(\a)\le k$, the type~$T$ of $\Ea$ depends only on
  $\a\pmod{p^{6(k+1)}}$, and each $\W_{T,k}$ is the union of
  $p^{20}N(T)$ classes modulo~$p^{6(k+1)}$.
\end{enumerate}
\end{cor}
For example, when $k=2$, knowing $\a\pmod{p^{18}}$ we can distinguish
between the cases $\lambda(\a)=0$, $\lambda(\a)=1$, $\lambda(\a)=2$ or
$\lambda(\a)\ge3$, and in all but the last case can also determine the
type of $\Ea$ from $\a\pmod{p^{18}}$; but to distinguish between
$\lambda(\a)=3$ and $\lambda(\a)\ge4$ we would need to know
$\a\pmod{p^{24}}$.

\section{General results relating $p$-adic densities and global densities}
\label{sec:general-density}
Our aim is to use the local density results of the previous section to
obtain global density results for integral Weierstrass equations.
This is straightforward if we only impose conditions at finitely many
primes, the conclusion being in general that the global density is
given, as one would expect, by the finite product of the local
densities.  This remains true when the local conditions are genuinely
$p$-adic, and not only given by congruences to finite powers of each
prime.  However, when we impose local conditions at all primes, the
passage from local to global densities is considerably more subtle.
Some general methods in this direction have been developed, notably
the ``Ekedahl Sieve'' introduced by Ekedahl in~\cite{Ekedahl-sieve},
and the approach of Poonen and Stoll in their paper
\cite{Poonen-Stoll} on the Cassels-Tate pairing on Abelian Varieties
(see also the shorter note~\cite{Poonen-Stoll-short} by the same
authors just on this issue).  For applications to the existence of
rational points on hypersurfaces, the results of Poonen and Voloch in
\cite{Poonen-Voloch} are often applicable, as for example in the case
of plane cubic curves in the paper~\cite{cubic-densities} of the first
author with Bhargava and Fisher.

In the prior work mentioned so far, only uniform densities were used;
in the case of quadrics in $n$ variables, treated by the first author
with Bhargava, Fisher, Jones, and Keating in~\cite{quadric-densities},
a different probability distribution was required at the real place,
requiring additional analysis there.  Further refinements to the
methods may be found in the work of Bhargava, for example in
\cite{Bhargava-sieve}. Furthermore, some specific cases not covered by
these have been handled individually, for example in the work of
Bhargava such as his results with Shankar and Wang on square-free
discriminants in \cite{bhargava2016squarefree}.

The results and approaches of the papers cited cannot easily be
applied directly in our situation, without additional discussion: for
example, we need the flexibility to adjust local conditions at
finitely many primes, and to introduce weights.  For this reason,
while our account in the rest of this section is firmly based on this
prior work, it is almost self-contained, the main exception being the
proof of the codimension~$2$ criterion of
Proposition~\ref{prop:PS-Lemma21}.

\subsection{Global densities I: finitely many $p$-adic conditions}
\label{subsec:globalI}
The standard definition of the \emph{uniform density} of a subset
$Z\subseteq\Z^d$ is as follows: we define the density of~$Z$ to be
\begin{equation}\label{eqn:def-uniform-density}
\begin{split}
  \rho(Z) &= \lim_{X\to\infty}\frac{\#\{\a\in Z\mid |a_i|\le
   X\ \forall i\}}{\#\{\a\in \Z^d\mid |a_i|\le
   X\ \forall i\}}\\
   &= \lim_{X\to\infty}(2X)^{-d}\#\{\a\in Z\mid |a_i|\le
   X\ \forall i\},
\end{split}
\end{equation}
if the limit exists.  Similarly, we define the \emph{upper density}
$\overline{\rho}(Z)$ and \emph{lower density} $\underline{\rho}(Z)$,
replacing the limit by $\lim\sup$ or $\lim\inf$ respectively.

More generally given any vector of positive real \emph{weights}
$\uk=(k_1,k_2,\dots,k_d)$ with sum $k=\sum_{i=1}^d k_i$, we can define
a \emph{weighted density}
\begin{equation}\label{eqn:def-weighted-density}
\begin{split}
\rho^{\uk}(Z) &= \lim_{X\to\infty}\frac{\#\{\a\in Z\mid |a_i|\le
X^{k_i}\ \forall i\}}{\#\{\a\in\Z^d\mid |a_i|\le
X^{k_i}\ \forall i\}}\\
 &= \lim_{X\to\infty}2^{-d}X^{-k}\#\{\a\in Z\mid |a_i|\le
X^{k_i}\ \forall i\}.
\end{split}
\end{equation}
Note that neither the existence nor the value of this limit is
affected if we scale the weight vector~$\uk$ by any positive real
factor. When all the weights are equal we recover the uniform density
as a special case.

We first determine the density of any subset~$Z\subseteq\Z^d$ defined
by congruence conditions at a finite set of primes, where it is given
by a simple counting formula not depending on the weights.  Let
$M\ge1$ and let $\Sigma\subseteq(\Z/M\Z)^d$ be an arbitrary subset.
One way to define such a set is locally, by choosing a finite set of
primes~$p$, a power $p^e$ of each, and a subset
$\Sigma_p\subseteq(\Z/p^e\Z)^d$.  Then set $M=\prod_p p^e$ and
$\Sigma=\prod_{p\mid M}\Sigma_p$, where we identify $\Z/M\Z$ with
$\prod_p\Z/p^e\Z$ by the Chinese Remainder Theorem.

Given~$\Sigma$, define $Z(M,\Sigma) = \{\a\in\Z^d\mid
\a\pmod{M}\in\Sigma\}$, and denote its weighted density by
$\rho^{\uk}(M,\Sigma)=\rho^{\uk}(Z(M,\Sigma))$ (or simply
$\rho(M,\Sigma)=\rho(Z(M,\Sigma))$ in the case of uniform density).

\begin{prop}\label{prop:modMdensity}
For all positive weights $\uk=(k_1,\dots,k_d)$ and all subsets
$\Sigma\subseteq(\Z/M\Z)^d$ we have
\[
\rho^{\uk}(M,\Sigma) = \frac{\#\Sigma}{M^d}.
\]
In particular, this is independent of~$\uk$.
\end{prop}
\begin{proof}
For $X>0$, set $Z(M,\Sigma;X) =
Z(M,\Sigma)\cap\prod_{i=1}^{d}[-X^{k_i},X^{k_i}]$.  For $1\le i\le d$,
set $b_i=\lfloor X^{k_i}/M\rfloor$, where, $\lfloor{x}\rfloor$ denotes
the integer $n$ such that $n\le x<n+1$. Then
\[
  M b_i \le X^{k_i} < M (b_i+1),
\]
and the interval $[-X^{k_i},+X^{k_i}]$ contains between $2b_i$ and
$2(b_i+1)$ complete sets of residue classes modulo~$M$.  Hence the box
$\Z^d\cap\prod_{i=1}^{d}[-X^{k_i},X^{k_i}]$ contains between $2^d\prod
b_i$ and $2^d\prod (b_i+1)$ complete sets of residue classes, each of
which contains $\#\Sigma$ elements of $Z(M,\Sigma)$.  So
$\#Z(M,\Sigma;X)$ satisfies
\[
2^d\#\Sigma\prod b_i \le \#Z(M,\Sigma;X) \le 2^d\#\Sigma\prod(b_i+1).
\]
Now $\prod b_i$ is approximately equal to $\prod X^{k_i}/M=X^{k}/M^d$
(where $k=\sum_{i=1}^dk_i$), so we approximate~$\#Z(M,\Sigma;X)$ by
$2^d\#\Sigma X^{k}/M^d$, and bound the error by noting that
$\prod(b_i+1)-\prod b_i$ is a sum of $2^d-1$ terms each bounded above
by~$X^{k-\min k_i}/M^{n}$ for some $n\le d-1$.  This gives
\[
\left|\frac{\#Z(M,\Sigma;X)}{2^dX^{k}}-\frac{\#\Sigma}{M^d}\right|
%\le \frac{\#\Sigma}{M^{d-1}}X^{-\min k_i},
 = O(X^{-\min k_i}),
 \]
with implied constant depending on the weights, $M$ and $\Sigma$ but
not~$X$. The result follows on letting $X\to\infty$.
\end{proof}

In the uniform case ($d=k$ and all $k_i=1$), we see from this proof
that
\[
\#Z(M,\Sigma;X) = \frac{\#\Sigma}{M^d}\cdot(2X)^d + O(X^{d-1}).
\]

The sets~$Z(M,\Sigma)$ considered so far are cut out by congruence
conditions at a \emph{finite} set of primes (those which divide~$M$),
each congruence being modulo some \emph{finite} power $p^e$.  Our next
step is to consider sets determined by a finite number of local
$p$-adic conditions.

Let $S$ be any set of primes (possibly including all primes).  We may
impose local $p$-adic conditions at all~$p\in S$ by specifying a
measurable subset $\U\subseteq\prod_{p\in S}\Z_p^d$.  Embedding~$\Z$
diagonally into $\prod_{p\in S}\Z_p$, this gives a subset $\U\cap\Z^d$
of~$\Z^d$, which we denote~$Z(\U)$.  When $S$ contains all primes we
denote $\prod_{p\in S}\Z_p$ by~$\Zhat$, and the local conditions are
determined by a measurable subset of~$\Zhat^d$.

For example, we may take $\U=\prod_{p\in S}U_p$, where for each $p\in
S$ we have a measurable subset $U_p\subseteq\Z_p^d$; we now have
\[
Z(\U) = \cap_{p\in S}(U_p\cap\Z^d) = \{\a\in\Z^d\mid
\a\in U_p\ \forall p\in S\}.
\]
Set
\begin{equation}\label{eqn:def-rho}
\rho^{\uk}(\U):=\rho^{\uk}(Z(\U)),
\end{equation}
and similarly define the
upper and lower densities~$\overline{\rho}^{\uk}(\U)$
and~$\underline{\rho}^{\uk}(\U)$.  The results which follow relate
these densities with the measure $\mu(\U)$, which in the special case
equals~$\prod_{p\in S}\mu(U_p)$, and we consider whether the equality
$\rho^{\uk}(\U)=\mu(\U)$ holds.  It is easy to show that the inequality
\[
\overline{\rho}^{\uk}(\U)\le\mu(\U)
\]
always holds. See Lemma~1.1 in~\cite{Ekedahl-sieve}, where Ekedahl
defines densities by counting the intersection with $[1,X]^d$ rather
than $[-X,X]^d$, but the result is the same.

From now on, we always take $\U$ to be a subset of the form
$\prod_{p\in S}U_p$ with $U_p\subseteq\Z_p^d$ measurable, and boundary
of measure zero: $\mu(\partial(U_p))=0$.  When $S$ is finite, the
density of $Z(\U)=\U\cap\Z^d$ always exists and equals the measure
$\mu(U)$.

\begin{prop}\label{prop:finiteUp}
Let $S$ be a finite set of primes, for each~$p\in S$ let
$U_p\subseteq\Z_p^d$ with $\mu(\partial(U_p))=0$, and set
$\U=\prod_{p\in S}U_p$.  Then for an arbitrary weight vector~$\uk$,
\[
   \rho^{\uk}(\U) = \prod_{p\in S}\mu(U_p).
\]
\end{prop}
\begin{rmk}
Note that this is essentially contained in the proof by Poonen and
Stoll in Lemma~20 of~\cite{Poonen-Stoll}, but there they include a
condition at the infinite place and do not have weights.  When the
infinite place is included we expect the density to depend on the
weights. By restricting our attention to sets defined by conditions
only at the finite places we obtain a simplification.
\end{rmk}

\begin{proof}
Set $M=\prod_{p\in S}p$.  For $\lambda\ge1$, define
\[
Y_{\lambda}=\{\a\in\Z^d\mid (\forall p\in S)(\exists \a_p\in U_p):
\a\equiv\a_p\pmod{p^{\lambda}}\}.
\]
Then
$\rho^{\uk}(Y_{\lambda})=\#\Sigma_{\lambda}/M^{d\lambda}$ by
Proposition~\ref{prop:modMdensity}, where $\Sigma_{\lambda}$ is the
reduction modulo~$M^{\lambda}$ of $Y_{\lambda}$, noting that
$Y_{\lambda}$ is a union of complete residue classes
modulo~$M^{\lambda}$.

The sets $Y_{\lambda}$ are nested ($Y_{\lambda+1}\subseteq
Y_{\lambda}$), their intersection is the closure of $Z(\U)$, which has
the same measure as $Z(\U)$ by our assumption on the boundary
measures.  Hence $\rho^{\uk}(\U) =
\lim_{\lambda\to\infty}\rho^{\uk}(Y_{\lambda}) =
\lim_{\lambda\to\infty}\#\Sigma_{\lambda}/M^{d{\lambda}} =
\prod_p\mu(U_p)$, where the last equality follows by the Chinese
Remainder Theorem and the definition of the $p$-adic measure.
\end{proof}

\subsection{Alternative choices of global density}
\label{subsec:alt-global}
The results of the previous subsection show that, provided that we
only consider subsets $Z\subseteq\Z^d$ defined by local conditions at
finitely many primes, weighted density does not depend on the weight
vector, so is the same as uniform density.

A different independence on the choice of a real probability
distribution of the global form of $p$-adic density was also noted
in~\cite{quadric-densities} by the first author with Bhargava
\textit{et al.}.  Let $D$ be a sufficiently
well-behaved\footnote{Piecewise smooth and rapidly decaying, in the
  sense that $D(x)$ and all its partial derivatives are $O(|x|^{-N})$
  for all $N>0$.} probability distribution on $\R^d$, so
$\int_{\R^d}D(\x)d\x=1$.  Then for $Z\subseteq\Z^d$, define
\[
\rho^D(Z) = \lim_{X\to\infty} \frac{\sum_{\a\in Z}D(\a/X)}{\sum_{\a\in\Z^d}D(\a/X)}.
\]
To recover our original (unweighted) definition take $D=U$, the
uniform distribution on the box $[-1,1]^d$.

It follows from \cite[\S2]{quadric-densities} that $\rho^D(Z)$ is
independent of the distribution~$D$.  A similar result (with a similar
proof) would hold for an analogous weighted definition of~$\rho^D(Z)$:
\[
\rho^{D,\uk}(Z) = \lim_{X\to\infty} \frac{\sum_{\a\in
    Z}D(\dots,a_i/X^{k_i},\dots)}{\sum_{\a\in\Z^d}D(\dots,a_i/X^{k_i},\dots)}.
\]
For the applications in this paper, we are not concerned with
constraints at the infinite place, so will not need this generality,
but it might be useful in other applications.  For example, we could
compute the density of elliptic curves over~$\R$ with positive and
negative discriminant, and hence include a fixed sign of the
discriminant in density results for elliptic curves over~$\Q$.  This
will depend on the distribution.

\subsection{Global densities II: infinitely many $p$-adic conditions}
\label{subsec:globalII}

We will closely follow the form of the Ekedahl Sieve used by Poonen
and Stoll, referring to their paper~\cite[\S9.3]{Poonen-Stoll} as
needed.  We find it convenient to discuss their results in terms of
the notion of an \emph{admissible family} to encapsulate the critical
condition in equation (10) of \cite[Lemma 20]{Poonen-Stoll} but not
given a name there.

Let $d\ge1$, and let $\U=\prod U_p\subseteq\Zhat^d$ be a subset
determined by a family of subsets $U_p\subseteq\Z_p^d$, one for each
rational prime~$p$.  As before, we suppose that each~$U_p$ is
measurable, and assume that the boundaries have measure zero.
For each~$M>0$ define
\begin{align*}
   Z_M(\U) &= \{\a\in\Z^d\mid \a\in U_p\ \text{for some prime $p>M$}\}\\
   &= \bigcup_{p>M}(U_p\cap\Z^d).
\end{align*}
For a positive weight vector~$\uk$, set
$\rho_M^{\uk}(\U)=\overline{\rho}^{\uk}(Z_M(\U))$.
\begin{defn}
The family $\U$ is \emph{admissible with respect to~$\uk$},
or \emph{$\uk$-admissible}, if
$\lim_{M\to\infty}\rho_M^{\uk}(\U)=0$.
\end{defn}
We will omit~$\uk$ from the notation when all the weights
are equal.

\begin{example}\label{ex:sq-free-ints}
Let $U_p=p^2\Z_p$ for all~$p$.  The associated set~$\U=\prod_pU_p$ is
admissible.

%%%%%%%%%%%%%%%%%%%%%%%%%%%%%%%%%%%%%%%%%%%%%%%%%%%%%%%%%%%%%%%%%%%%%%%%%%%
\begin{comment}
  For each~$p$, the density of integers which are divisible by~$p^2$
is~$1/p^2$ (by Proposition~\ref{prop:modMdensity}).  Now
$Z_M(\U)=\cup_{p>M}p^2\Z$.  By definition, we have
\begin{align*}
\rho_M(\U)
&= \limsup_{X\to\infty}\frac{1}{2X}\#\left(Z_M(\U)\cap[-X,X]\right) \\
&= \limsup_{X\to\infty}\frac{1}{2X}\#\bigcup_{p>M}(p^2\Z\cap[-X,X]) \\
&\le \limsup_{X\to\infty}\frac{1}{X}\sum_{p>M}\#(p^2\Z\cap[1,X]) \\
&\le \limsup_{X\to\infty}\left(\sum_{M<p\le X^{1/2}}\frac{1}{p^2}\right),
\end{align*}
since $a\in Z_M(\U)\cap[1,X]\implies p^2\le|a|\le X$ (for some prime
$p>M$), so the term $\#(p^2\Z\cap[1,X])$ is the integer part
of~$X/p^2$, hence is zero unless $p\le X^{1/2}$.  Hence
\[
\rho_M(\U) \le  \sum_{p>M}\frac{1}{p^2},
\]
which tends to zero as~$M\to\infty$ since $\sum_p1/p^2$
converges. Hence $\U$ is admissible.
\end{comment}
%%%%%%%%%%%%%%%%%%%%%%%%%%%%%%%%%%%%%%%%%%%%%%%%%%%%%%%%%%%%%%%%%%%%%%%%%%%

\end{example}

Proof of admissibility in this example uses the fact that
$\sum_p\mu(U_p)$ converges; however, this is not sufficient for $\prod
U_p$ to be admissible. The next example, where $\mu(U_p)=0$ for all
$p$ but still $\U$ is not admissible, was shown to us by Michael
Stoll.

\begin{example}
For $n\ge1$ let $p_n$ be the $n$th prime, and define $U_{p_n}=\{n\}$,
the singleton set.  Then $\mu(U_p)=0$ for all $p$, but $Z_M(\U)$
contains all positive integers~$n$ except for the finitely many for
which $p_n \le M$, so its density for each~$M$ is the same as the
density of the set of all positive integers, namely $1/2$.  So the
limit is not zero and $\U=\prod U_p$ is not admissible.
\end{example}

It will be useful to have simple sufficient criteria for a family to
be admissible.  First we note the following easy consequences of the
definition.
\begin{lem}\label{lem:admiss}
\begin{enumerate}
\item Let $\U'=\prod U_p'$ be a second family such that $U_p'=U_p$ for
  all but finitely many primes~$p$.  Then $\U$ is
  $\uk$-admissible if and only if $\U'$ is
  $\uk$-admissible, for any weight vector~$\uk$.
\item Let $\U'=\prod U_p'$ be a second family with $\U'\subseteq\U$
  (that is, $U_p'\subseteq U_p$ for all~$p$).  Then
  $\uk$-admissibility of~$\U$ implies
  $\uk$-admissibility of~$\U'$.
\end{enumerate}
\end{lem}
\begin{proof}
The first statement holds, since $Z_M(\U)=Z_M(\U')$ for all $M$ greater
than the largest prime $p$ for which $U_p\not=U_p'$; the second is
clear, since $Z_M(\U')\subseteq Z_M(\U)$.
\end{proof}

Let $S$ be any set of primes.  Define
\[
   \rho^{\uk}(\U,S) = \rho^{\uk}(\{\a\in\Z^d\mid \a\in U_p\iff p\in S\}),
\]
the density of the set of integer vectors which lie in the
distinguished subset $U_p$ precisely for the primes in $S$.  Taking
$S$ to be the set of all primes, we have
$\rho^{\uk}(\U,S)=\rho^{\uk}(\U)$ as defined in~\eqref{eqn:def-rho}.
In what follows we will use the subsets $U_p$ to encode conditions to
be avoided, so that the density we are most interested in is
$\rho^{\uk}(\U,\emptyset)$, which we hope under certain conditions to
equal $\prod_p(1-s_p)$, where~$s_p=\mu(U_p)$.

The result from~\cite{Poonen-Stoll} which we will use is the
following: for admissible families, the density exists and equals the
measure, so we have the desired product formula.
\begin{prop}\label{prop:PS-Lemma20}
Let $\U=\prod U_p$ be an admissible family with respect to the weight
vector~$\uk$, with $s_p=\mu(U_p)$ and $\mu(\partial U_p)=0$.  Then
$\sum_{p}s_p$ converges, and for every finite set $S$ of primes,
\begin{equation}\label{eqn:PS-formula}
   \rho^{\uk}(\U,S) = \prod_{p\in S}s_p \prod_{p\notin S}(1-s_p).
\end{equation}
In particular, the density of the set of~$\a\in\Z^d$ which do not lie
in~$U_p$ for any prime~$p$ is $\rho^{\uk}(\U,\emptyset) =
\prod_p(1-s_p)$, and $\rho^{\uk}(\U,S) = 0$ if $S$ is infinite.
\end{prop}
\begin{proof}
Replacing $U_p$ by its complement in $\Z_p^d$ for $p \in S$ gives
another admissible family by Lemma~\ref{lem:admiss}, and the general
formula \eqref{eqn:PS-formula} follows from the same result for this
latter family. Hence we may assume that $S=\emptyset$.

%%   The case of general finite~$S$ follows from the special case
%% $S=\emptyset$ by Lemma~\ref{lem:admiss}, replacing $U_p$ by its
%% complement in~$\Z_p^d$ for $p\in S$.  Hence we may assume that
%% $S=\emptyset$.

To ease notation we omit the superscript~$\uk$, writing $\rho$ for
$\rho^{\uk}$.

Assume that $U_p=\emptyset$ for all $p>M$ for some~$M$. Let $U'_p$ be
the complement of~$U_p$ in~$\Z_p^d$.  Now
\[
\rho(\U,\emptyset) = \rho(\prod_{p\le M}U'_p) = \prod_{p\le M}(1-s_p),
\]
by Proposition~\ref{prop:finiteUp}.  This gives~\eqref{eqn:PS-formula}
since $s_p=0$ for all~$p>M$.

Hence, we have in general for each $M>0$,
\[
\rho(\prod_{p\le M} U'_p) = \prod_{p\le M}(1-s_p).
\]
Now $Z(\prod_p U'_p) \subseteq Z(\prod_{p\le M} U'_p)$; the sets
$Z(\prod_{p\le M} U'_p)$ form a decreasing nested sequence whose
intersection as $M\to\infty$ is $Z(\prod_p U'_p)$.  The complement is
\begin{align*}
Z(\prod_{p\le M} U'_p) \setminus Z(\prod_{p} U'_p) &= \{\a\in\Z^d\mid
\text{$\a\notin U_p$ for all $p\le M$, and $\a\in U_p$ for some
  $p>M$}\}\\
&\subseteq Z_M(\U),
\end{align*}
whose density tends to zero by the admissibility condition.  Hence
\[
\rho(\U,\emptyset) = \rho(Z(\prod_{p} U'_p)) =
\lim_{M\to\infty}\rho(\prod_{p\le M} U'_p) =
\lim_{M\to\infty}\prod_{p\le M}(1-s_p) = \prod_p(1-s_p)
\]
as required.
\end{proof}

Note that it follows from Proposition~\ref{prop:PS-Lemma20} that the
density $\rho^{\uk}(\U,S)$ is independent of the weight vector~$\uk$,
being equal to a product which does not depend on~$\uk$, provided
that~$\U$ is~$\uk$-admissible.

\begin{ex-cont}[Example 1 continued]
Proposition~\ref{prop:PS-Lemma20}, together with the admissibility
statement of Example~\ref{ex:sq-free-ints}, implies the well-known
result that the density of the set of square-free integers is
$1/\zeta(2)$.  Since $\U=\prod p^2\Z_p$ defines an admissible family,
with $s_p=1/p^2$, the density of square-free integers is
$\rho(\U,\emptyset)=\prod_p(1-1/p^2)=1/\zeta(2)$.
\end{ex-cont}

Closed subschemes of $\Z^d$ of codimension at least~$2$ determine
admissible conditions.  The following is the simplest example:

\begin{example}
  The set of coprime pairs $(a,b)\in\Z^2$ has density $1/\zeta(2)$.

%%%%%%%%%%%%%%%%%%%%%%%%%%%%%%%%%%%%%%%%%%%%%%%%%%%%%%%%%%%%%%%%%%%%%%%%%%%
\begin{comment}
Let $S=\{(a,b)\in\Z^2\mid\gcd(a,b)=1\}$, cut out by the local sets
$\U=\prod U_p$, where $U_p=p\Z_p^2=\{(a,b)\in\Z_p^2\mid a\equiv
b\equiv0\pmod{p}\}$.  Then
\[
Z_M(\U) = \bigcup_{p>M}p\Z^2.
\]
Let $Z_M^*(\U)$ be the subset of~$Z_M(\U)$ where $ab\not=0$; this has
the same density as~$Z_M(\U)$, since the coordinate axes clearly have
density zero.  If $(a,b)\in Z_M^*(\U)\cap[-X,X]^2$ then $p\le X$, since
$|a|$ and $|b|$ are positive multiples of~$p$ and ${}\le X$.  Hence
\[
Z_M^*(\U)\cap[-X,X]^2 = \bigcup_{M<p\le X}\left((p\Z\setminus\{0\})^2\cap[-X,X]^2\right).
\]
The $p$th set in this finite union is finite, with cardinality exactly
$4[X/p]^2$, so
\[
\frac{1}{4X^2}\#\left(Z_M^*(\U)\cap[-X,X]^2\right) \le \sum_{M<p\le
  X}\frac{1}{p^2}.
\]
Letting $X\to\infty$ gives
\[
\rho_M(\U) \le \sum_{M<p}\frac{1}{p^2},
\]
so $\lim_{M\to\infty}\rho_M(\U)=0$ and $\U$ is admissible. Hence
$\rho(S) = \prod_p(1-\mu(U_p)) = \prod_p(1-1/p^2) = 1/\zeta(2)$.
\end{comment}
%%%%%%%%%%%%%%%%%%%%%%%%%%%%%%%%%%%%%%%%%%%%%%%%%%%%%%%%%%%%%%%%%%%%%%%%%%%
\end{example}

This is a special case (with $d=2$, $f=X_1$ and $g=X_2$) of the
following much more general result of Poonen and Stoll (see
\cite[Lemma 21]{Poonen-Stoll}), which will be crucial for our
applications in the next section.  Note that the proof given in
\cite{Poonen-Stoll} simply states that it follows immediately from a
result of Ekedahl (Theorem 1.2 of~\cite{Ekedahl-sieve}) applied to the
closed subscheme of the affine scheme $\A_{\Z}^d$ cut out by $f=g=0$,
making use of the fact that the subscheme has codimension~$2$.
However, while it is clear that Ekedahl's theorem implies that the
product formula holds in this situation, for our applications in the
next section we need to know that $\prod U_p$ is admissible, so that
we can adjust the $p$-adic condition at $p=2$ and $p=3$.  It is hard
to extract this precise statement from Ekedahl's proof, but the
necessary details have been supplied by Bhargava
in~\cite[Theorem~3.3]{Bhargava-sieve} which we use instead.

\begin{prop}\label{prop:PS-Lemma21}
Let $f,g\in\Z[X_1,\dots,X_d]$ be coprime polynomials.  Let
$\U=\prod U_p$ where
\[
   U_p = \{\a\in\Z_p^d\mid f(\a)\equiv g(\a)\equiv0\pmod{p}\}.
\]
Then $\U$ is $\uk$-admissible, for all weight vectors~$\uk$.
\end{prop}

\begin{proof}
The $\Z$-scheme~$Y$ cut out by $f=g=0$ has codimension~$2$.  In the
case of uniform weights, we may apply Bhargava's estimate
\cite[Theorem~3.3]{Bhargava-sieve} (with $n=d$, $k=2$, $B=[-1,1]^d$,
$r=X$) to see that the cardinality of $Z_M(\U)$ is $O(X^d/(M\log M) +
X^{d-1})$, and hence $\rho_M(\U) = O(1/(M\log M))$ which tends to~$0$
as~$M\to\infty$.

For the general case, we note that (as remarked by Bhargava \textit{et
  al.} in \cite[p.~4]{bhargava2016squarefree}), his result
\cite[Theorem~3.3]{Bhargava-sieve} also holds in the weighted case.
\end{proof}

\begin{rmk}
It has been observed by Bhargava (see the remarks on page~4
of~\cite{bhargava2016squarefree} by Bhargava \textit{et al.} for a
similar observation) that among families $\U=\prod U_p$ with
$\mu(U_p)=O(1/p^2)$, it is necessary to distinguish between those
where $U_p$ is defined by two independent ``mod~$p$'' conditions, as
in Proposition~\ref{prop:PS-Lemma21}, and those defined by a single
``mod~$p^2$'' condition.  An example of the latter is to take a single
square-free polynomial $f\in\Z[X_1,\dots,X_d]$ and define $U_p$ to be
the subset of $\a\in\Z_p^d$ where $f(\a)\equiv0\pmod{p^2}$, in order
to determine the density of the set of $\a\in\Z^d$ such that $f(\a)$
is square-free.  In the former case, the family is $\uk$-admissible
for any weights~$\uk$, but in the latter case additional work is
needed in order to establish $\uk$-admissibility for suitable~$\uk$
and~$f$, to conclude that the global weighted density is the product
of local densities.  The example treated in
\cite{bhargava2016squarefree}, of monic integral polynomials with
square-free discriminant, is of the latter type.

Moreover, as discussed by Bhargava in \cite[\S1.3]{Bhargava-sieve}, a
general result that the global density exists for all square-free~$f$
and is given by the product formula, is closely related to the
$abc$-conjecture.  In \cite{Granville-ABC}, Granville proved that the
$abc$-conjecture implies the result for polynomials in one variable
and arbitrary degree (the case of quadratics is easier, and for cubics
was established by Hooley in~\cite{Hooley-power-free}, these cases
being unconditional).  In~\cite{Poonen-squarefree}, Poonen proves this
also for multivariable polynomials, using an unconditional reduction
to the univariate case.
\end{rmk}

The simplest example where we require\footnote{See, however, the
  remark at the end of this section.} non-uniform weights to establish
$\uk$-admissibility is the following.  This example is closely related
to the density of monic cubics in $\Z[X]$ with square-free
discriminant, and we give details in the following example as a
similar technique will be required in the next section when we
consider the density of integral Weierstrass equations with
square-free discriminant.

\begin{example}\label{ex:sqfree3disc}
Let $S=\{(a,b)\in\Z^2\mid a^3-b^2\ \text{is square-free}\}$, cut out
by the local conditions $\U=\prod U_p$ where $U_p=\{(a,b)\in\Z_p^2\mid
a^3\equiv b^2\pmod{p^2}\}$.  We show that~$\U$ is $\uk$-admissible for
the weights~$\uk=(2,3)$, and hence that $S$ has density given by the
product formula
\[
    \rho^{\uk}(S) = \prod_p\left(1-2/p^2+1/p^3\right).
\]
Write $U_p$ as the disjoint union~$U_p'\cup U_p''$, where
$U_p'=p\Z_p^2$ and $U_p''=\{(a,b)\in\Z_p^2\mid p\nmid ab, p^2\mid
a^3-b^2\}$.  Set $\U'=\prod U_p'$ and $\U''=\prod U_p''$.

\begin{lem}
  $\mu(U_p)=2/p^2-1/p^3$.
\end{lem}
\begin{proof}
Clearly $\mu(U_p')=1/p^2$.  To compute $\mu(U_p'')$ it suffices to
consider $a,b$ modulo~$p^2$ and note that there is a bijection between
$\{(a,b)\in((\Z/p^2\Z)^*)^2\mid a^3=b^2\}$ and $(\Z/p^2\Z)^*$ given by
$(a,b)\mapsto b/a$ with inverse $t\mapsto(t^2,t^3)$.  Hence
$\mu(U_p'')=\varphi(p^2)/p^4=1/p^2-1/p^3$.
\end{proof}

In the language of \cite{bhargava2016squarefree} by Bhargava
\textit{et al.}, $a^3-b^2$ is ``strongly divisible'' by $p^2$ for
$(a,b)\in U_p'$ but only ``weakly divisible'' for $(a,b)\in U_p''$.
The proof of $\uk$-admissibility for $\U'=\prod_p U_p$ is easier, and holds
for arbitrary weights, while that for $\U''=\prod_p U_p''$ is more subtle,
and only works when $3k_1\le2k_2$.  The choice~$(2,3)$ for the weights
is natural, considering that the discriminant of the
cubic~$X^3-3aX+2b$ is $108(a^3-b^2)$.  Hence, apart from the
conditions at $2$ and~$3$ requiring adjustment, the density of~$S$
would give the density of monic cubics with square-free discriminant,
with weights matching the natural ones for the coefficients of a monic
univariate polynomial.  The main result in
\cite{bhargava2016squarefree} gives the density of monic polynomials
in~$\Z[X]$ with square-free discriminant as the product of local
densities, in arbitrary degree, the result for degree~$3$ being that
the density is $\frac{1}{2}\prod_{p\ge3}(1-2/p^2-1/p^3)$, agreeing
with our formula~$\rho(S)=\prod_{p}(1-2/p^2-1/p^3)$ except for the
local density at~$2$.  In~\cite{bhargava2016squarefree}, the weights
used for monic cubics $X^3+a_1X^2+a_2X+a_3$ are $\uk=(1,2,3)$,
consistent with our choice of weights $\uk=(2,3)$ for~$S$.

In showing that $\U'$ and $\U''$ are admissible, we may ignore the
set~$Z_0$ of pairs $(a,b)\in\Z^2$ with $a^3=b^2$ (that is, pairs of
the form $(t^2,t^3)$ for some $t\in\Z$), as well as those for which
$ab=0$, since these form a subset of density zero.

We first show that $\U'$ is admissible, with arbitrary positive
weights~$k_1,k_2$. (This also follows from
Proposition~\ref{prop:PS-Lemma21}.)  For this we must estimate the
cardinality of the set
\[
\bigcup_{p>M}\{(a,b)\in p\Z^2\setminus Z_0: |a|\le X^{k_1}, |b|\le
X^{k_2}\},
\]
divide by $4X^{k_1+k_2}$ and let $X\to\infty$ to obtain an estimate
for the tail density~$\rho_M(\U')$.  The $p$th set in the union has
cardinality $O((X^{k_1}/p)(X^{k_2}/p))=O(X^{k_1+k_2}/p^2)$, and is
empty for $p>X^{\min(k_1,k_2)}$, so the union has cardinality
\[
O(\sum_{M<p\le X^{\min(k_1,k_2)}}X^{k_1+k_2}/p^2).
\]
Dividing by~$4X^{k_1+k_2}$ and letting $X\to\infty$, this is bounded
above by~$\sum_{p>M}1/p^2$ and hence tends to~$0$ as $M\to\infty$.

Now we show that $\U''$ is $\uk$-admissible for $\uk=(2,3)$; the same
argument is valid whenever $3k_1\le2k_2$, but not for equal weights.  We
estimate the cardinality of the set
\[
\bigcup_{p>M}\{(a,b)\in\Z^2\setminus Z_0: |a|\le X^2, |b|\le X^3,
p\nmid ab, p^2\mid a^3-b^2\},
\]
and show that, after dividing by~$4X^5$ and letting $X\to\infty$, the
resulting tail density~$\rho_M(\U'')$ tends to~$0$ as~$M\to\infty$.
The $p$th set in this union is empty for $p>\sqrt{2}X^3$, since
$p^2\le|a^3-b^2|\le 2X^6$.  Let $p$ be a prime with
$M<p\le\sqrt{2}X^3$. For each integer~$a$ with $p\nmid a$, the number
of solutions~$b$ to the congruence $b^2\equiv a^3\pmod{p^2}$ is
either~$2$ or~$0$, according to whether $a$ is a quadratic residue or
not modulo~$p$, as it follows from Hensel's Lemma that (since $p\nmid
b$) each solution modulo~$p$ lifts uniquely to a solution
modulo~$p^2$.  Since each residue class modulo~$p^2$
has~$2X^3/p^2+O(1)$ representatives~$b$ in the interval~$[-X^3,X^3]$,
the number of pairs $(a,b)$ to be counted (for each~$a$)
is~$4X^3/p^2+O(1)$, or zero.  Hence the cardinality of the set above
is at most
\[
\sum_{M<p\le\sqrt{2}X^3}(2X^2)(4X^3/p^2+O(1)).
\]
The main term is
\[
8X^5\sum_{M<p\le\sqrt{2}X^3}(1/p^2),
\]
which after dividing by~$4X^5$ and letting $X\to\infty$ is
$2\sum_{p>M}1/p^2$, which tends to~$0$ as~$M\to\infty$ as required.

Each of the remaining terms is of size $O(X^2)$, and the number of
terms is at most~$\pi(\sqrt{2}X^3)=O(X^3/\log X)$ by the Prime Number
Theorem, so their sum is~$O(X^5/\log X)$.  Dividing by~$4X^5$ and
letting $X\to\infty$, we see that the contribution of these error
terms is negligible.

If the weights are $(k_1,k_2)$ with $k_1/k_2>2/3$, then the total
contribution of the error terms in the last part of the proof is no
longer negligible.
\end{example}

\begin{rmk}
Although the proof we have given here for the density of square-free
values of $a^3-b^2$ does not work with equal weights, the result also
holds in this case, but the proof is considerably deeper.  We are
grateful to Manjul Bhargava for explaining this to us.

Instead of $a^3-b^2$ we consider square-free values of~$-4 a^{3} - 27
b^{2}$, the discriminant of the cubic~$x^{3} + a x + b$.  Embed the
space of such cubics with integer coefficients into the larger space
of all binary cubic forms over~$\Z$, on which $\GL_2(\Z)$ acts,
leaving the discriminant invariant.  In this larger space, ordering
cubic forms by their height (the maximum absolute value of the
coefficients), one can show that the density of those with square-free
discriminant is the expected product of local densities, by showing
that the associated tail densities tend to zero.
%% Davenport showed that the number of cubic forms whose discriminant is
%% less than~$X$ and divisible by~$p^2$ is ${}\le X/p^2$, not just
%% $X/p^2+O(1)$, leading to a simple proof of admissibility in this
%% situation.  (Recall that in the preceding example, the difficulty only
%% arose from the $O(1)$ error terms, since the number of terms was too
%% large for the sum of these to be negligible.)
Finally, the number of solutions to the Thue equation $F(x,y)=1$ for a
binary cubic form~$F$ over~$\Z$ is bounded by~$10$ (Evertse gave the
bound~$12$ in 1983 in \cite{evertse1983representation}, and this was
improved to~$10$ by Bennett in 2001 in
\cite{bennett2001representation}).  Hence each~$\GL_2(\Z)$-orbit of
binary cubic forms contains at most~$10$ with leading coefficient~$1$,
and possibly fewer with coefficients of the form~$1,0,a,b$, so the tail
density estimates for binary cubic forms also apply to square-free
discriminants of cubic polynomials~$x^3+ax+b$.
\end{rmk}

\section{Global densities for elliptic curves}\label{sec:global}
We now apply the results of the previous section in dimension~$d=5$,
together with the local densities determined in
Section~\ref{sec:localI}, to determine global densities of integral
Weierstrass equations satisfying certain combinations of local
conditions.

Recall from Section~\ref{sec:localI} that $\W(\Z)=\Z^5$ is the space
of all Weierstrass equations with coefficients in~$\Z$, and now
consider elliptic curves over $\Q$ defined by long integral
Weierstrass equations $E_{\a}$ for $\a = (a_1,a_2,a_3,a_4,a_6) \in
\W(\Z)$.  In common with other work on density results for elliptic
curves, we use weighted densities with weights
\[
  \uk = (1/12,2/12,3/12,4/12,6/12) = (1/12,1/6,1/4,1/3,1/2),
\]
so for $X>0$ we define
\[
\E(X)=\{\a\in\W(\Z)\mid |a_i|\le X^{i/12}\ \text{for $i=1,2,3,4,6$}\}.
\]
We have $\#\E(X)\sim 32X^{4/3}$, as the sum of the weights is~$4/3$.

For any subset $U\subseteq\W(\Z)$ recall that the (weighted) density
$\rho^{\uk}(U)$ of~$U$ was defined
(see~\eqref{eqn:def-weighted-density}) as
\[
 \rho^{\uk}(U) = \lim_{X\to\infty}\frac{\#\E(X)\cap U}{\#\E(X)}.
\]
By homogeneity, the density is unchanged if we use the weight
vector~$(1,2,3,4,6)$ instead of $(1/12,1/6,1/4,1/3,1/2)$, as we did in
the Introduction \eqref{eqn:def-density-intro}.  Since the weight
vector will remain fixed throughout this section, we simplify notation
by writing $\rho(U)$ for $\rho^{\uk}(U)$ in what follows. However,
apart from the result about square-free discriminants
(Theorem~\ref{thm:square-free-disc}), it is not hard to see that the
results of this section are independent of the weights, since we
impose no condition at the infinite place and otherwise rely only on
Proposition~\ref{prop:PS-Lemma21}.

\subsection{Global densities with a condition at a single prime}

Fix a prime $p$. For each local type $T(p)$ of elliptic curves over
$\Q_p$, let $\W_{T(p)}(\Z)=\W_{T(p)}(\Z_p)\cap\W(\Z)$ and
$\W_{T(p)}^M(\Z)=\W_{T(p)}^M(\Z_p)\cap\W(\Z)$.
The \emph{global density} $\rho_{T(p)}^{\Z}$ of type $T(p)$ can now be
defined as the density of~$\W_{T(p)}(\Z)$.  There are two versions, the
second one including the minimality condition at $p$.
\begin{defn}\label{def:rhoTU}
Set $\rho_{T(p)}^{\Z} = \rho(\W_{T(p)}(\Z))$ and $\rho_{T(p)}^{\Z
  M} = \rho(\W_{T(p)}^M(\Z))$.
\end{defn}
In words, $\rho_{T(p)}^{\Z}$ is the density of the set of integral
Weierstrass equations defining elliptic curves with reduction
type~$T(p)$ at the prime~$p$, while $\rho_{T(p)}^{\Z M}$ is the
density of the set of integral Weierstrass equations which are minimal
at the prime~$p$ and models of elliptic curves with reduction
type~$T(p)$.

These global densities are equal to the corresponding $p$-adic
densities, both with and without the minimal condition at~$p$:

\begin{thm}\label{thm:rhoTU}
Let $T(p)$ be one of the finite $p$-adic types (as listed in
Proposition~\ref{prop:TAtable}, depending only on $\a$ modulo~$p^6$).  Then
\[
  \rho_{T(p)}^{\Z} = \rho_{T(p)}
\]
  and
\[
  \rho_{T(p)}^{\Z M} = \rho_{T(p)}^M.
\]
\end{thm}

\begin{proof}
Write $T=T(p)$.  The first statement follows from
Proposition~\ref{prop:finiteUp}, with $U_p=\W_T(\Z_p)$ and
$U_q=\W(\Z_q)$ for all primes~$q\not=p$.

By definition we have $\rho_{T}^{\Z M} = \rho(\W_{T}^M(\Z))$, and the
latter is equal to~$\rho_{T}^M$ by Proposition~\ref{prop:modMdensity}
with modulus~$p^{6}$, giving the second statement.
\end{proof}

\begin{example}
The density of elliptic curves over~$\Q$ with good reduction at~$2$
(with no restrictions at any other primes) is
\[
   (1-2^{-1})/(1-2^{-10}) = 2^9/(2^{10}-1) = 512/1023 \approx 50.0\%.
\]
\end{example}

\begin{example}
The density of elliptic curves over~$\Q$ with additive reduction of
type $\III^*$ at~$5$ (with no restrictions at any other primes) is
\[
   (5^2-5)/(5^{10}-1) = 1/406091.
\]
\end{example}

\subsection{Global densities with conditions at finitely many primes}

Let $S$ be a finite set of primes, and for each $p\in S$ fix a finite
reduction type~$T(p)$.  Applying Proposition~\ref{prop:muWk} with
Proposition~\ref{prop:modMdensity} and Proposition~\ref{prop:finiteUp}
we immediately obtain the following.
\begin{thm}\label{thm:finite-global}
Let $S$ be any finite set of primes, and for each $p\in S$ let $T(p)$
be a finite reduction type.
\begin{enumerate}
\item
The density of integral Weierstrass equations which for all~$p\in S$
are minimal at~$p$ with reduction type~$T(p)$ is $\prod_{p\in
  S}\rho_{T(p)}^M$.
\item
The density of elliptic curves over~$\Q$ whose reduction type at~$p$
is~$T(p)$ for all~$p\in S$ is $\prod_{p\in S}\rho_{T(p)}$.
\end{enumerate}
\end{thm}

\begin{example}
The density of elliptic curves over~$\Q$ with good reduction at
both~$2$ and~$3$ (with no restrictions at any other primes) is
\[
    2^9(2-1)3^9(3-1)/(2^{10}-1)(3^{10}-1) = 839808/2516921 \approx 33.37\%.
\]
\end{example}

\begin{example}
Let $p_1$, $p_2$ and $p_3$ be distinct primes.  The density of
elliptic curves over~$\Q$ with good reduction at~$p_1$, multiplicative
reduction at~$p_2$ and additive reduction at~$p_3$ (with no restrictions at any other primes) is
\[
   \left(\frac{1-p_1^{-1}}{1-p_1^{-10}}\right)\left(\frac{p_2^{-1}-p_2^{-2}}{1-p_2^{-10}}\right)\left(\frac{p_3^{-2}-p_3^{-10}}{1-p_3^{-10}}\right).
\]
\end{example}

\subsection{Global densities with conditions at infinitely many primes}

To obtain density results with conditions at infinitely many primes,
we may use Proposition~\ref{prop:PS-Lemma20}, provided that the
excluded sets $U_p\subset\W(\Z_p)$ form an admissible family.  The
previous subsection dealt with the simplest case where almost all
$U_p$ were empty.

Recall the standard invariants $c_4,c_6\in\Z[a_1,a_2,a_3,a_4,a_6]$ of
a Weierstrass model $E_{\a}$.  As elements
of~$\Z[a_1,a_2,a_3,a_4,a_6]$, they are both irreducible (being linear
in $a_4$ and $a_6$ respectively), and coprime.  The results in this
subsection follow from the following.

\begin{lem}\label{lem:c4c6}
  For each prime~$p$, define $U_p\subseteq\W(\Z_p)$ by
\[
U_p = \{\a\in\W(\Z_p)\mid c_4(\a)\equiv c_6(\a)\equiv0\pmod{p}\}.
\]
Then the family $\U=\prod U_p$ is admissible.
\end{lem}
\begin{proof}
Since $c_4$ and~$c_6$ are coprime, we may apply
Proposition~\ref{prop:PS-Lemma21}.
\end{proof}

For $p\ge5$, the condition $c_4\equiv c_6\equiv0\pmod{p}$ is equivalent
to the Weierstrass model~$\Ea$ being non-minimal or of bad additive
reduction\footnote{Note that for $p=2$ and $p=3$ one can have good
  reduction when $p\mid c_4$ and $p\mid c_6$, for example
  \lmfdbec{11}{a}{1} for $p=2$ and \lmfdbec{17}{a}{1} for $p=3$.
  Also, $c_4$ and $c_6$ are not coprime as polynomials over~$\F_2$ or
  $\F_3$.}, so, for $p\ge5$, we have $U_p=U_p'$ where
\[
   U_p' = \W(\Z_p) \setminus \left(\W_{\I_0}^M \cup
   \W_{\I_{\ge1}}^M\right).
\]
Since $\mu(U_p') = 1- (\rho_{I_0}^M + \rho_{I_{\ge1}}^M) = 1/p^2$ for
all primes~$p$, it follows that $\mu(U_p)=1/p^2$ for all~$p\ge5$.
One may also check that $\mu(U_p)=1/p$ for $p=2,3$, using $c_4\equiv
a_1^4$ and $c_6\equiv a_1^6\pmod{2}$, and $c_4\equiv (a_1^2+a_2)^2$
and $c_6\equiv -(a_1^2+a_2)^3\pmod{3}$, but we will not need these
values.

Recall that an elliptic curve is called \emph{semistable at a
  prime~$p$} if its reduction type is either good (type~$\I_0$) or
multiplicative (type~$\I_{\ge1}$), and~\emph{semistable} if it is
semistable at all primes.

\begin{thm}\label{thm:global-semistable}
  \
  \begin{enumerate}
\item
The density of integral Weierstrass equations which are minimal models
of semistable elliptic curves is $1/\zeta(2)\approx 60.79\%$.
\item
The density of semistable elliptic curves over~$\Q$ is
$\zeta(10)/\zeta(2)\approx 60.85\%$.
\end{enumerate}
\end{thm}
\begin{proof}
Let $\U=\prod U_p$ and~$\U'=\prod U_p'$ be as above. Since $\U$ is
admissible by Lemma~\ref{prop:PS-Lemma21}, so is $\U'$ by
Lemma~\ref{lem:admiss}.  Also, $\mu(\U_p')=1/p^2$ for all~$p$, by
Proposition~\ref{prop:TAtable} (as noted above).  Taking $S=\emptyset$
in Proposition~\ref{prop:PS-Lemma20} gives the density stated, since
$\prod_p(1-1/p^2)=1/\zeta(2)$.

For the second part we let $U_p''$ be the set of Weierstrass models of
curves with additive reduction.  This is a subset of~$U_p$, since
$U_p$ includes not only these models but also non-minimal models of
curves with good or multiplicative reduction.  Now the local density
of curves with good or multiplicative reduction is
$(1-p^{-2})/(1-p^{-10})$, so $\mu(U_p'') = 1 -
(1-p^{-2})/(1-p^{-10})$.  Applying Proposition~\ref{prop:PS-Lemma20}
again yields the desired density as $\prod_p(1-\mu(U_p'')) =
\prod_p(1-p^{-2})/(1-p^{-10}) = \zeta(10)/\zeta(2)$.
\end{proof}

We can obtain further global density results by changing the local
conditions at any finite set of primes, provided that we know the
associated local densities.  The only constraint on results provable
in this way is therefore that, at all but finitely many primes, the
condition we impose is that of semistability, \textit{i.e.}, good or
multiplicative reduction.  As in the two parts of
Theorem~\ref{thm:global-semistable}, if we also impose conditions of
minimality at all primes, this will not affect the convergence
criteria, merely dividing the global density by
$\prod_p(1-p^{-10})^{-1}=\zeta(10) = \pi^{10}/93555 \approx
1.000994575$.  This establishes the following.

\begin{thm}\label{thm:global-semistable+}
Let $S$ be any finite set of primes, and for each $p\in S$ let $T(p)$
be a finite reduction type.
\begin{enumerate}
\item
The density of integral Weierstrass equations which are global minimal
models of elliptic curves over~$\Q$ whose reduction type at~$p$
is~$T(p)$ for all~$p\in S$, and which are semistable at all other
primes, is
\[
   \zeta(2)^{-1} \prod_{p\in S}\rho_{T(p)}/(1-p^{-2}).
\]
\item
The density of elliptic curves over~$\Q$ whose reduction type at~$p$
is~$T(p)$ for all~$p\in S$ and which are semistable at all other
primes is
\[
   \zeta(10)\zeta(2)^{-1} \prod_{p\in S}\rho_{T(p)}/(1-p^{-2}).
\]
\end{enumerate}
\end{thm}

\subsection{Curves with square-free discriminant}
A Weierstrass equation has square-free discriminant if and only if it
is minimal and of reduction type $\I_0$ or $\I_1$.  These have local
density $1-1/p$ and $(p-1)^2/p^3$, by Propositions~\ref{prop:TAtable}
and~\ref{prop:typesIm} respectively, so the local density of those
with square-free discriminant is~$1-{2}/{p^2}+{1}/{p^3}$.  Hence the
set $U_p$ of Weierstrass equations with discriminant divisible
by~$p^2$ has local density $2/p^2-1/p^3$.  By comparison with the case
of square-free discriminants of monic cubic polynomials (see
Example~\ref{ex:sqfree3disc} in the previous section), we expect
$\U=\prod U_p$ to be admissible.  This is indeed the case, provided
that we use appropriate weights, as specified at the start of this
section.

\begin{thm}\label{thm:square-free-disc}
\begin{enumerate}
\item
The density of integral Weierstrass equations whose discriminant is
square-free is
\[
   \prod_{p}\left(1-\frac{2}{p^2}+\frac{1}{p^3}\right) \approx 42.89\%.
\]
\item
The density of elliptic curves over~$\Q$ whose minimal discriminant is
square-free is
\[
   \zeta(10)\prod_{p}\left(1-\frac{2}{p^2}+\frac{1}{p^3}\right) \approx 42.93\%.
\]
\end{enumerate}
\end{thm}

For $p\not=2$, the local density $(1-{2}/{p^2}+{1}/{p^3})$ is exactly
the same as that of monic cubic polynomials over $\Z_p$ with
square-free discriminant (see \cite{bhargava2016squarefree} and
\cite[Theorem 6.8]{Ash2007}).  Hence, by \cite[Theorem
  1.1]{bhargava2016squarefree}, the theorem states that the
probability that a random integral Weierstrass equation has
square-free discriminant is (after taking the discrepancy for $p=2$
into account) equal to $5/4$ times the probability that a random monic
integral cubic polynomial has square-free discriminant.

\begin{proof}
The proof follows the argument given in Example~\ref{ex:sqfree3disc}
above, taking the additional variables into account.  We again write
$U_p$ as a disjoint union $U_p=U_p'\cup U_p''$, where
\[
   U_p' = \W(\Z_p) \setminus  \W_{\I_{\ge0}}^M,
\]
is the set of Weierstrass equations with bad additive reduction at~$p$
or non-minimal at~$p$, and
\[
   U_p'' = \W_{\I_{\ge2}}^M
\]
is the set of Weierstrass equations with multiplicative reduction
at~$p$ of Type~$\I_m$ for some~$m\ge2$.  Admissibility of $\U'=\prod
U_p'$ has already been established in the proof of
Theorem~\ref{thm:global-semistable}, so we consider admissibility
of~$\U''=\prod U_p''$.  Ignoring $p=2$ and~$3$, as we may by
Lemma~\ref{lem:admiss}, the condition for belonging to $U_p''$ is that
$p\nmid c_4,c_6$ but $p^2\mid\Delta$, or equivalently $p^2\mid
c_4^3-c_6^2$.  This is a ``mod~$p^2$ condition'', in contrast to
membership of $U_p'$ which is a ``mod~$p$ condition''.

Regarding~$\Delta$ as a polynomial in~$a_6$ with coefficients
in~$\Z[a_1,a_2,a_3,a_4]$, it has degree~$2$ with leading
coefficient~$-432=-2^43^3$ and discriminant~$c_4^3$. (Note that
$c_4\in\Z[a_1,a_2,a_3,a_4]$ does not depend on~$a_6$.) Hence for each
fixed $(a_1,a_2,a_3,a_4)\in\Z^4$ with $p\nmid c_4$, there are at
most~$2$ solutions for $a_6\pmod{p}$ to the
congruence~$\Delta\equiv0\pmod{p}$, each of which lifts to a unique
solution to~$\Delta\equiv0\pmod{p^2}$.

Secondly, each term in~$\Delta$ has weight~$12$ when we give $a_i$
weight~$i$, so for~$\a$ bounded by $|a_i|\le X^{i/12}$, each monomial
appearing in~$\Delta$ is bounded by~$X$, and hence~$|\Delta| \le BX$,
where $B$ is the sum of the absolute values of the coefficients
of~$\Delta$.  (In fact, $B=1714$, but the actual numerical coefficient
is unimportant.)  It follows that if $\a$ satisfies the weighted
bounds, $p^2\mid\Delta$ and $\Delta\not=0$, then $p\le(BX)^{1/2}$.

To compute the tail density~$\rho_M^{\uk}(\U'')$, we must estimate the
cardinality of the set
\[
\bigcup_{p>M}\{\a\in\Z^5: |a_i|\le X^{i/12}, p\nmid c_4, p^2\mid\Delta\},
\]
and we may ignore~$\a$ with $\Delta=0$ as these have zero density.
The $p$th set in this union is empty unless $p\le(BX)^{1/2}$.  For
each $p$ below this bound, the number of
$4$-tuples~$(a_1,a_2,a_3,a_4)$ satisfying the bounds
is~$O(X^{10/12})=O(X^{5/6})$, and each $4$-tuple determines at most
two values of~$a_6\pmod{p^2}$, hence $O(X^{1/2}/p^2)+O(1)$ values
of~$a_6$ also satisfying $|a_6|\le X^{1/2}$.  Adding over all~$p$ with
$M<p\le(BX)^{1/2}$, the main term is
\[
O(X^{5/6})\sum_{M<p\le(BX)^{1/2}}O(X^{1/2}/p^2) =
O(X^{4/3})\sum_{M<p\le(BX)^{1/2}}1/p^2,
\]
contributing at most $\sum_{p>M}1/p^2$ to the tail density.  Each
error term is~$O(X^{5/6})$ and the number of terms is
$O(\pi((BX)^{1/2})) = O(X^{1/2}/\log X)$, so the total error
is~$O(X^{4/3}/\log X)$ which is $o(X^{4/3})$ and hence negligible.

This completes the proof that~$\U''$ is admissible, and the rest of
the statement of the Theorem follows as before.
\end{proof}

\begin{rmk}
  It is perhaps worth noting what are the properties of the
  discriminant polynomial $\Delta(a_1,a_2,a_3,a_4,a_6)$ which ensure
  that the above proof works.

  Firstly, it is \emph{isobaric} with respect to certain positive
  weights of the variables~$a_i$ (meaning that each monomial has the
  same weight).  We use these weights of the variables (scaled
  by~$1/12$ to make the total weight of~$\Delta$ equal to~$1$, but
  that is unimportant) as the weights used to define the density.

  Secondly, we used the fact that $\Delta$ has degree only~$2$ in one
  of the variables, $a_6$.  Careful examination of the proof above
  reveals that, in order to show that the error terms were negligible,
  it was crucial that the exponent $1/2$ on the bound for this
  variable matched the exponent on the bound on~$p$, which in turn
  came from the fact that our condition was that $\Delta$ was
  square-free.

  As with square-free values of the discriminant of a cubic
  polynomial~$x^3+ax+b$, it is possible that
  Theorem~\ref{thm:square-free-disc} also holds using equal weights on
  the coefficients~$a_i$, but we have not tried to prove this.  One
  approach might be to embed the space of Weierstrass cubics in the
  larger set of ternary cubic forms over~$\Z$, for which this result
  is known: see~\cite{bhargava2016squarefree}.

  We would also expect the methods used here to be able to establish
  the density of monic integer quartic polynomials whose discriminant
  is \emph{cube-free}, using the natural weights rather than equal
  weights, since that discriminant has degree~$3$ in the constant
  coefficient, but that determining the density of square-free
  discriminants of quartic (and higher degree) monic integer
  polynomials would be harder; indeed, the methods used
  in~\cite{bhargava2016squarefree} to evaluate this (in arbitrary
  degree) are much deeper.
\end{rmk}

\subsection{Curves with prime-power conductor (or discriminant)}

Finally in this section, we consider elliptic curves with a single
prime of bad reduction.

Fix $X>0$, and consider first elliptic curves with a single prime
$p<X$ of bad multiplicative reduction, good reduction at all other
primes $q<X$, and no restriction at primes $q>X$.  That is, we
consider elliptic curves of conductor $N = pN'$ where $N'$ has no
prime factors less than $X$.  This set has density
\[
\sum_{p\le X}\left((1/p-1/p^2)\prod_{q\le X, q\not=p}(1-1/q)\right)  =
\left(\sum_{p\le X} 1/p\right)  \left(\prod_{q\le X}(1-1/q)\right).
\]
As $X\to\infty$, the first factor $\sum_{p\le X}(1/p) \sim \log\log
X$, while $\prod_{q\le X}(1-1/q) \sim e^{-\gamma}/\log X$, where
$\gamma$ is Euler's constant.  Hence the density is~$O(\log\log X/\log
X)$, and tends to $0$ as $X\to\infty$.

Hence the density of elliptic curves with prime discriminant is also
zero, as these are a subset of those with prime conductor.

A small modification of this argument applies to curves of prime power
conductor (equivalently, prime power discriminant).  For each~$X$, the
set of curves with precisely one prime~$p\le X$ of bad reduction has
density
\[
\sum_{p\le X}\left(1/p \prod_{q\le X, q\not=p}(1-1/q)\right) = \left( \sum_{p\le
  X}1/(p-1)\right) \left( \prod_{q\le X}(1-1/q)\right).
\]
Since $1/(p-1)-1/p=1/p(p-1)$ and $\sum1/p(p-1)$ converges, the
asymptotics are unchanged.

\section{Local densities II}\label{sec:localII}
In this section we extend the local density results of
Section~\ref{sec:localI} to include the distribution of conductor
exponents~$f_p$ and Tamagawa numbers~$c_p$, for each type of
reduction.  Consequent global results may be obtained using the
methods of Sections~\ref{sec:general-density} and~\ref{sec:global}.

The results here are all obtained by following in detail the steps of
Tate's Algorithm, as originally given in~\cite{tate1975algorithm}.
Our methods are similar to those employed by Papadopoulos in
\cite{papadopoulos}, where he establishes congruence conditions on the
Weierstrass coefficients~$a_i$ for each Kodaira reduction type.  As
Papadopoulos observes, for $p\ge5$ the type is completely determined
by the valuations of the invariants $c_4$, $c_6$ and~$\Delta$; for
$p=3$ one can make use of the coefficients $b_i$, while for $p=2$ one
is forced to consider all the $a_i$.  Since the expression for
$\Delta$ as a polynomial in the $a_i$ has~$26$ terms, this would be
tiresome to do by hand, and we use computer algebra to assist us.  The
reader may find \Sage\ code to verify the claims made in this section
at \cite{sagecode}.  The main differences between the results of this
section and those of Papadopoulos are that we quantify each step in
order to find the $p$-adic density of each case, while on the other
hand Papadopoulos works in the more general context of a local field
and not just $\Q_p$ itself.

Throughout, $p$ will denote a fixed prime; in the results and proofs
we often need to consider $p=2$ and $p=3$ separately.

All curves with good reduction at~$p$ have $f_p=0$ and $c_p=1$. It is
well-known that the density of Weierstrass equations which have good
reduction is~$1-1/p$. The first author first learned the following
fact from Hendrik Lenstra (who showed him a different proof from the
one which follows), but as we do not know a suitable reference we
include a proof here.

\begin{lem}\label{lem:sing-count}
Let $q$ be a prime power.  Of the $q^5$ Weierstrass equations
over~$\F_q$, precisely $q^4$ are singular.
\end{lem}
\begin{proof}
Weierstrass equations define irreducible cubic curves, and by Bezout's
Theorem, they can have at most one singular point, which is not the
unique point at infinity, and hence is one of the $q^2$ points in the
affine plane.  For each of these, the number of equations having the
specified point as its singular point is the same (by translation), so
it suffices to count equations for which $P=(0,0)$ is singular.  Now
$P$ lies on the curve if and only if $a_6=0$, and then $P$ is singular
if and only if $a_3=a_4=0$, so there are $q^2$ equations for which $P$
is singular, and $q^4$ singular equations in all.
\end{proof}

Recall from Section~\ref{sec:localI} the notation
\[
\W(v_1,v_2,v_3,v_4,v_6) = \{\a\in\W(\Z_p)\mid \ordp(a_i)\ge
v_i\ \text{for~$i=1,2,3,4,6$}\}.
\]

\begin{prop}\label{prop:good-density}
The density of Weierstrass equations over~$\Z_p$ which have good
reduction is
\[
  \mu(\W(0,0,0,0,0\mid\ordp(\Delta)=0)) = 1-1/p.
\]
\end{prop}
\begin{proof}
Immediate from Lemma~\ref{lem:sing-count}.
\end{proof}

For the bad reduction types the distributions of $f_p$ and
$c_p$ are as follows.

\begin{thm}[Distribution of conductor exponents and Tamagawa numbers
    by reduction type] \label{thm:fp-cp-count} Within each bad
  reduction type, whose density is given by
  Proposition~\ref{prop:TAtable}, the relative densities of each
  possible conductor exponent and Tamagawa number are as follows.
  Where two possibilities are given for the Tamagawa number, the
  density is split equally between them.

\begin{enumerate}

\item Multiplicative reduction types, all~$p$:
\[
\begin{tabular}{cccccc}
\hline
Type && $f_p$ & $c_p$ & relative density & absolute density \\
\hline
$\I_m$ & each $m\ge1$ & $1$ &  & $(p-1)/p^m$ & $(p-1)^2/p^{m+2}$ \\
\hline
$\I_m$ & split & $1$ & $m$ & $1/2$ & $(p-1)/(2p^2)$ (total for all~$m$)\\
       & non-split, $m$ even & $1$ & $2$ & $1/(2(p+1))$ &
$(p-1)/(2p^2(p+1))$ (total for all even~$m$) \\
       & non-split, $m$ odd & $1$ & $1$ & $p/(2(p+1))$  &
$(p-1)/(2p(p+1))$ (total for all odd~$m$) \\
\hline
\end{tabular}
\]

\item Additive reduction types:

\begin{itemize}
\item[$p\ge5$:]
\[
\begin{tabular}{cccc}
\hline
Type & $f_p$ & $c_p$ & relative density \\
\hline
$\II$,$\II^*$ & $2$ & $1$ & $1$ \\
\hline
$\III$,$\III^*$ & $2$ & $2$ & $1$ \\
\hline
$\IV$, $\IV^*$ & $2$ & $1$ or $3$ & $1$ \\
\hline
$\I_0^*$ & $2$ & $1$ & $(p+1)/(3p)$ \\
        & $2$ & $2$ & $1/2$ \\
        & $2$ & $4$ & $(p-2)/(6p)$ \\
\hline
$\I_m^*$ & $2$ & $2$ or $4$ & $1$ \\
\hline
\end{tabular}
\]

\item[$p=3$:]
\[
\begin{tabular}{cccc}
\hline
Type & $f_p$ & $c_p$ & relative density \\
\hline
$\II$,$\II^*$ & $3$ & $1$ & $2/3$ \\
          & $4$ & $1$ &  $2/9$ \\
          & $5$ & $1$ &  $1/9$ \\
\hline
$\III$,$\III^*$ & $2$ & $2$ & $1$ \\
\hline
$\IV$,$\IV^*$ & $3$ & $1$ or $3$ & $2/3$ \\
          & $4$ & $1$ or $3$ & $2/9$ \\
          & $5$ & $1$ or $3$ & $1/9$ \\
\hline
$\I_0^*$ & $2$ & $1$ & $4/9$ \\
        & $2$ & $2$ & $1/2$ \\
        & $2$ & $4$ & $1/18$ \\
\hline
$\I_m^*$ & $2$ & $2$ or $4$ & $1$ \\
\hline
\end{tabular}
\]

\item[$p=2$:]
\[
\begin{tabular}{cccc}
\hline
Type & $f_p$ & $c_p$ & relative density \\
\hline
$\II$     & $4$ & $1$ & $1/2$ \\
          & $6$ & $1$ & $3/8$ \\
          & $7$ & $1$ & $1/8$ \\
\hline
$\II^*$   & $3$ & $1$ & $1/2$ \\
          & $4$ & $1$ & $1/4$ \\
          & $6$ & $1$ & $1/4$ \\
\hline
$\III$,$\III^*$ & $3$ & $2$ & $1/2$ \\
           & $5$ & $2$ & $1/4$ \\
           & $7$ & $2$ & $1/8$ \\
           & $8$ & $2$ & $1/8$ \\
\hline
$\IV$,$\IV^*$ & $2$ & $1$ or $3$ & $1$ \\
\hline
$\I_0^*$ & $4$ & $1$ or $2$ & $1/2$ \\
       & $5$ & $1$ or $2$ & $1/4$ \\
       & $6$ & $1$ or $2$ & $1/4$ \\
\hline
$\I_m^*$ & $3$ & $2$ or $4$ & $1/2$ \\
       & $4$ & $2$ or $4$ & $1/4$ \\
      & $5$ & $2$ or $4$ & $1/16$ \\
      & $6$ & $2$ or $4$ & $1/8$ \\
      & $7$ & $2$ or $4$ & $1/16$ \\
\hline
\end{tabular}
\]
\end{itemize}
\end{enumerate}
\end{thm}

In the proof we use the following elementary counting lemmas; the
second is Lemma~3 in~\cite{cubic-densities}.

\begin{lem}\label{lem:quad-count}
Let $q$ be a prime power.  Of the $q^2$ monic quadratics
$f\in\F_q[X]$,
\begin{itemize}
\item $q$ have a double root;
\item $q(q-1)/2$ have distinct roots in~$\F_q$;
\item $q(q-1)/2$ have conjugate roots in~$\F_{q^2}$.
\end{itemize}
\end{lem}

\begin{lem}\label{lem:cubic-count}
Let $q$ be a prime power.  Of the $q^3$ monic cubics
$g\in\F_q[X]$,
\begin{itemize}
\item $q^2$ have a multiple root, of which
\begin{itemize}
\item[$\cdot$] $q$ have a triple root (necessarily in $\F_q$);
\item[$\cdot$] $q(q-1)$ have a double root and a single root (both in $\F_q$);
\end{itemize}
\item $q^3-q^2$ have distinct roots, of which
\begin{itemize}
\item[$\cdot$] $q(q-1)(q-2)/6$ have distinct roots in~$\F_q$;
\item[$\cdot$] $q^2(q-1)/2$ have one root in~$\F_q$ and two conjugate roots
  in~$\F_{q^2}$;
\item[$\cdot$] $q(q^2-1)/3$ have conjugate roots in~$\F_{q^3}$.
\end{itemize}
\end{itemize}
\end{lem}

\subsection{Proof of Theorem~\ref{thm:fp-cp-count}}

During the course of the proof, we will fill in details which were
only sketched in the proof of Proposition~\ref{prop:TAtable}.

We follow the steps of Tate's Algorithm.  Recall from
Section~\ref{sec:localI} the notation $\T=\T(\Z_p)=\{\tau(r,s,t)\mid
r,s,t\in \Z_p\}$.  We also set $n=\ordp(\Delta)$.

Initially there are no conditions except integrality of the
coefficients, so we start in $\W(0,0,0,0,0)$.  At each step, we
{either} exit the algorithm based on a divisibility test; {or}, we
divide into subcases.  The exit criteria always occur with probability
$1/p$.  The division into subcases is always into~$p$ subcases, except
at the beginning where there are $p^2$ subcases, one for each
possibility for the singular point mod~$p$.  The subcases occur with
equal probabilities, and the relative densities within each subcase
are independent of the specific subcase: for example, when there is
bad reduction, each of the $p^2$ points in the affine $\F_p$-plane is
equally likely to be the unique singular point, and the densities of
each bad reduction type do not depend on which point is singular.

\subsubsection*{Good reduction} The exit condition is
$n=0$: then $f_p=0$ and $c_p=1$.  This occurs with probability
$1-1/p$, by Proposition~\ref{prop:good-density}.  Otherwise (with
probability $1/p$), we divide into $p^2$ equiprobable subcases,
proceeding with the case where the point $(0,0)$ is singular
(modulo~$p$).  So now $\a\in\W(0,0,1,1,1)$.

Since $\Delta$ is invariant under the whole translation group~$\T$,
the exit condition is well-defined.  We claim that the stabiliser
of $\W(0,0,1,1,1)$ in~$\T$ is $\TT101$. In one direction this is
obvious from the transformation formulas for~$\tau(r,s,t)$, which for convenience we
recall here:
\begin{align*}
  a_1'-a_1 &= R_1 = 2s\\
  a_2'-a_2 &= R_2 = -sa_1+3r-s^2\\
  a_3'-a_3 &= R_3 = ra_1+2t\\
  a_4'-a_4 &= R_4 = -sa_3+2ra_2-(t+rs)a_1+3r^2-2st\\
  a_6'-a_6 &= R_6 = ra_4+r^2a_2+r^3-ta_3-t^2-rta_1.
\end{align*}
If $p$ divides all of $r,t,a_3, a_4, a_6$, then it divides $a_3',
a_4', a_6'$ also.  Conversely, suppose that $\tau(r,s,t)$ preserves
$\W(0,0,1,1,1)$.  Then $R_3\equiv R_4\equiv R_6\equiv0$, and
  \[
  r^3 \equiv (rs-t)R_3 + rR_4 -2R_6 \pmod{p}%\tag{*}
  \]
implies $r\equiv0$; then $-t^2\equiv R_6\equiv0$ implies $t\equiv0$.

\subsubsection*{Multiplicative reduction} Given
$\a\in\W(0,0,1,1,1)$, the exit condition $v(b_2)=0$ is that
$f=y^2+a_1y-a_2$ has distinct roots modulo~$p$.  By
Lemma~\ref{lem:quad-count}, this occurs with probability~$1-1/p$, so
\begin{equation} \label{eqn:B_I1}
\mu(\W(0,0,1,1,1\mid\ordp(b_2)=0) = \frac{p-1}{p}\mu(\W(0,0,1,1,1)=(p-1)/p^4.
\end{equation}
Note that this condition is invariant under $\TT101$, since $b_2' =
b_2+12r\equiv b_2\pmod{p}$.  In this case, $f_p=1$ and the type is
$\I_m$ where $m=n$ (${}=\ordp(\Delta)$), while the value of~$c_p$
depends on the parity of~$m$ and on whether the reduction type is
split or non-split, which in turn depends on whether or not the roots
of~$f$ lie in $\F_p$.

In the split case, $c_p=m$, with density $\frac{1}{2}(p-1)^2/p^{m+2}$,
for each~$m\ge1$.  Relative to the total density of Type~$\I_{\ge1}$, this
is $(p-1)/2p^m$.

In the non-split case, $c_p=1$ for odd~$m$, with total density
$\frac{1}{2}\sum_{k=0}^{\infty}(p-1)^2/p^{2k+3}=(p-1)/{2p(p+1)}$,
while $c_p=2$ for even~$m$, with total density
$\frac{1}{2}\sum_{k=1}^{\infty}(p-1)^2/p^{2k+2}=(p-1)/{2p^2(p+1)}$.
Relative to the total density of Type~$\I_{\ge1}$, these are $p/2(p+1)$ and
$1/2(p+1)$ respectively.

Otherwise, $\ordp(b_2)\ge1$ and we move on to the types of additive
reduction; after another transformation taking the double root of
$f\pmod{p}$ to~$0$, we have $\a\in\W(1,1,1,1,1)$.  This translation
has the form $\tau(0,s,0)\in\TT101$ with $s$ unique modulo~$p$, so
the stabiliser of $\W(1,1,1,1,1)$ is cut down from $\TT101$ to
$\TT111$.

For $\a\in\W(v_1,v_2,v_3,v_4,v_6)$ we follow Tate's notation
in~\cite{tate1975algorithm} and write $a_{i,v_i}=p^{-v_i}a_i$. In the
course of the proof, there are many claims of the form
$\Delta\equiv*\pmod{p^k}$, where the right-hand side is in
$\Z[a_1,a_2,a_3,a_4,a_6]$, and the claim is made under the assumption
that $p^{v_i}\mid a_i$ for $1\le i\le6$; such claims can all be
verified by expanding the difference of both sides and checking that
every term has valuation at least~$k$.  In all cases, the coefficients
of every term are not divisible by any primes other than~$2$ and~$3$,
which explains why these primes often need separate treatment.  While
it would be possible to give a simpler proof for $p\ge5$ only, in term
of the invariants~$c_4$ and~$c_6$, we will treat all primes in as
uniform a way as possible, for clarity.  All these claims may be
checked using the \Sage\ code at \cite{sagecode}.

\subsubsection*{Additive reduction, Type~$\II$} Given
$\a\in\W(1,1,1,1,1)$, the exit condition for Type~$\II$ is
$\ordp(a_6)=1$.  This is well-defined since for $\a\in\W(1,1,1,1,1)$
and $\tau\in\TT111$ we have $\ordp(a_6'-a_6)\ge2$.

In this case we have $c_p=1$ and $f_p=n$.  Given $\a\in\W(1,1,1,1,1)$,
we find that
\(
\Delta \equiv -2^43^3a_6^2 \pmod{p^3};
\)
so $n\ge2$, and when the exit condition holds (so that $\ordp(a_6)=1$), we
have $n=2$, provided that $p\ge5$.

For $p=3$, we have
\(
\Delta \equiv -a_4^3 \pmod{3^4},
\)
so $n\ge3$, with $n=3\iff\ordp(a_4)=1$, which has
relative probability~$2/3$.  Otherwise, $\a\in\W(1,1,1,2,\veq1)$, with
\(
\Delta \equiv -3a_2^2a_6 \pmod{3^5},
\)
so $n\ge4$, with $n=4\iff \ordp(a_2)=1$, since $\ordp(a_6)=1$;
this case happens with relative probability $(1/3)(2/3)=2/9$.
Otherwise, $\a\in\W(1,2,1,2,\veq1)$, with
\(
\Delta \equiv -3^3a_6^2 \pmod{3^6},
\)
so $n=5$ with the remaining relative probability~$1/9$.

For $p=2$, we have
\(
\Delta\equiv a_3^4 \pmod{2^5},
\)
so $n\ge4$, and $n=4\iff \ordp(a_3)=1$, which has
relative probability~$1/2$.  Otherwise, $\a\in\W(1,1,2,2,\veq1)$, with
\(
\Delta \equiv a_1^4a_4^2-2^4a_6^2 \equiv a_1^4a_4^2-2^6 \pmod{2^7},
\)
so $n\ge6$, with $n=6\iff \ordp(a_1^4a_4^2)=6$; this case happens when
either $\ordp(a_1)\ge2$ or $\ordp(a_4)\ge3$, so with
relative probability~$(1/2)(3/4)=3/8$.  Assuming that both $\ordp(a_1)=1$
and~$\ordp(a_4)=2$, we find that
\(
\Delta \equiv 2^7 \pmod{2^8},
\)
so $n=7$ with the remaining relative probability~$1/8$.

Otherwise, we have $\a\in\W(1,1,1,1,2)$, with unchanged stabiliser
$\TT111$.

\subsubsection*{Additive reduction, Type~$\III$} Given
$\a\in\W(1,1,1,1,2)$, the exit condition for Type~$\III$ is
$\ordp(a_4)=1$.  This is well-defined since for $\a\in\W(1,1,1,1,2)$
and $\tau\in\TT111$ we have $v(a_4'-a_4)\ge2$.

In this case we have $c_p=2$ and $f_p=n-1$.  Now we have
\(
\Delta \equiv -2^6a_4^3 \pmod{p^4},
\)
so $n=3$ and $f_p=2$ for $p\ge3$, since $\ordp(a_4)=1$.

For $p=2$ and $\a\in\W(1,1,1,\veq1,2)$, we have
\(
\Delta \equiv a_3^4 \pmod{2^5},
\)
so $n\ge4$, with $n=4\iff \ordp(a_3)=1$, which happens
with relative probability~$1/2$.  Otherwise, $\a\in\W(1,1,2,\veq1,2)$,
and we have
\(
\Delta \equiv 2^2a_1^4 \pmod{2^7},
\)
so $n\ge6$, with $n=6\iff\ordp(a_1)=1$; this case has relative
probability~$(1/2)(1/2)=1/4$.  Otherwise, $\a\in\W(2,1,2,\veq1,2)$,
and we have
\(
\Delta \equiv 2^8 (a_{2,1}^2 + a_{3,2}^4 + a_{6,2}^2) \pmod{2^9},
\)
so $n\ge8$, with $n=8\iff2\nmid a_{2,1}+a_{3,2}+a_{6,2}$; this case has
relative probability~$(1/4)(1/2)=1/8$.  Finally, assuming that
$a_6\equiv2a_2+a_3\pmod{8}$ (so that $a_6'\equiv a_2'+a_3'\pmod2$), we find that
\(
\Delta\equiv2^9\pmod{2^{10}},
\)
so $n=9$ with the remaining relative probability~$1/8$.

The relative probabilities for $n=4,6,8$, and~$9$ (respectively,
$f_2=3,5,7$, and~$8$) are therefore $1/2,1/4,1/8$, and~$1/8$.

Otherwise, we have $\a\in\W(1,1,1,2,2)$, with unchanged stabiliser
$\TT111$.

\subsubsection*{Additive reduction, Type~$\IV$} Given
$\a\in\W(1,1,1,2,2)$, the exit condition for Type~$\IV$ is that the
quadratic~$f=y^2+a_{3,1}y-a_{6,2}$ has distinct roots modulo~$p$, or
equivalently that $\ordp(b_6)=2$. This condition is well-defined,
since for $\tau\in\TT111$ we have $\ordp(b_6'-b_6)\ge3$.  Using
Lemma~\ref{lem:quad-count} again, we have
\begin{equation}\label{eqn:B_IV}
\mu(\W(1,1,1,2,2\mid\ordp(b_6)=2)) = \frac{p-1}{p}\mu(\W(1,1,1,2,2)) =
(p-1)/p^8.
\end{equation}
Now $f_p=n-2$, and $c_p=1$ or~$3$, according to whether the roots of
$f$ are in~$\F_p$ or not, which have relative probability $1/2$ each;
it remains to determine the possible values of the discriminant
valuation~$n$ and their relative densities.

For $\a\in\W(1,1,1,2,2\mid\ordp(b_6)=2)$, we have
\(
\Delta\equiv-3^3b_6^2\pmod{p^5},
\)
so for $p\not=3$ we have $n=4$ and $f_p=2$.

For $p=3$, we have
\(
\Delta \equiv -a_2^3b_6 \pmod{3^6},
\)
so $n\ge5$, with $n=5\iff \ordp(a_2)=1$.  Otherwise,
$\a\in\W(1,2,1,2,2\mid\ordp(b_6)=2)$, and we have
\(
\Delta\equiv b_4^3\pmod{3^7}.
\)
Note that $b_4=a_1a_3+2a_4$, so $\ordp(b_4)\ge2$.  Hence $n\ge6$, and
$n=6\iff\ordp(b_4)=2\iff a_4\not\equiv a_1a_3\pmod{3^3}$. Assuming
that $a_4\equiv a_1a_3\pmod{3^3}$, so that $\ordp(b_4)\ge3$, we find that
\(
\Delta\equiv -3^3b_6^2\pmod{3^8},
\)
so $n=7$.  Thus for $p=3$, we have $n=5, 6$, or~$7$ and $f_3=3, 4$
or~$5$ with relative probabilities $2/3$, $2/9$, and~$1/9$
respectively.

Otherwise, $\ordp(b_6)\ge3$, so the quadratic $y^2+a_{3,1}y-a_{6,2}$ has a
repeated root.  A transformation~$\tau$ in a unique coset of $\TT112$
in~$\TT111$ takes the root to~$0$ and hence the coefficients into
$\W(1,1,2,2,3)$, with stabiliser $\TT112$.

\subsubsection*{Additive reduction, Type~$\I_0^*$} Given
$\a\in\W(1,1,2,2,3)$, the exit condition for Type~$\I_0^*$ is
$\ordp(\disc(g))=6$, where $g=x^3+a_2x^2+a_4x+a_6$.  Equivalently, the
condition is that $g_1(x) = g(px)/p^3 = x^3+a_{2,1}x^2+a_{4,2}x+a_{6,3}$ should
have distinct roots in $\overline{\F}_p$, since
$\ordp(\disc(g_1))=\ordp(\disc(g))-6$.  Note that after transforming
the equation by $\tau_{r,s,t}\in\TT112$, $g_1(x)$ becomes
$g_1(x+r/p)$, so the condition is well-defined.

Now, $c_p$ is equal to one more than the number of roots of $g_1$ in
$\F_p$.  By Lemma~\ref{lem:cubic-count}, this number is $0$, $1$
or~$3$ with relative probabilities $(p+1)/(3p)$, $1/2$ and
$(p-2)/(6p)$ respectively.  We have
\(
  \Delta\equiv16\disc(g) \pmod{p^7},
\)
so for $p\not=2$, the exit condition implies $n=6$, and then
$f_p=n-4=2$.

Now let $p=2$.  For $\a\in\W(1,1,2,2,3)$, we have
\(
\disc(g)\equiv a_6^2-a_2^2a_4^2 \pmod{2^7},
\)
so the exit condition implies $a_6\not\equiv a_2a_4\pmod{2^7}$.
Assuming that $a_6\not\equiv a_2a_4 \pmod{2^7}$, we find that
\(
\Delta \equiv a_1^4a_4^2-a_3^4 \pmod{2^9},
\)
so $n\ge8$, with $n=8 \iff a_{3,2}\not\equiv a_{1,1}a_{4,2}\pmod2$.  Assuming
further that $a_{3,2}\equiv a_{1,1}a_{4,2}\pmod2$, we have
\(
\Delta\equiv2^8a_1\pmod{2^{10}},
\)
so $n\ge9$, with $n=9$ if and only if $\ordp(a_1)=1$.   When
$\ordp(a_1)\ge2$, then also $\ordp(a_3)\ge3$, and these imply that
\(
\Delta\equiv 2^{10}\pmod{2^{11}},
\)
giving $n=10$.

The preceding analysis shows that for type $\I_0^*$ curves when $p=2$
we have $n=\ordp(\Delta)=8$, $9$, or~$10$, and respectively $f_2=4$,
$5$, or~$6$, with relative probabilities $1/2$, $1/4$, and~$1/4$.

\subsubsection*{Additive reduction, Type~$\I_m^*$, $m\ge1$}
The exit condition for type~$\I_0^*$ fails when $g_1(x)$ has a
repeated root modulo~$p$.  We can move this root to zero using a
transform in a unique coset of~$\TT212$ in $\TT112$, after which
$\a\in\W(1,1,2,3,4)$, with stabiliser now $\TT212$.  The condition
for type $\I_m^*$ is that the repeated root is only a double root,
which (after the transform) is that $\ordp(a_2)=1$.

Looking at the details of Tate's algorithm in this case, it proceeds
in a sequence of substeps: at each substep the value of~$m$ is
incremented; there is an exit condition that a monic quadratic mod~$p$
has distinct roots; and the value of~$c_p$ depends on whether this
quadratic has roots in~$\F_p$ (in which case $c_p=4$) or not
($c_p=2$).  So, overall, each of these two values occurs in half the
cases, by Lemma~\ref{lem:quad-count}.  Moreover, the stabiliser index
increases by a factor of~$p$ at each stage, since when the quadratic
has a double root we can move it to~$0$ with a transform in a uniquely
determined coset of an index~$p$ subgroup of the current stabiliser.

We have $f_p=n-m-4$.  We treat separately the cases $p\ge3$, where we
will see that $f_p=2$ always, and $p=2$.  The case $p\ge3$ is
well-known (see Kraus~\cite{Kraus-additive} or
Kobayashi~\cite{Kobayashi}), but we include the details here since the
analysis is similar to that required for~$p=2$.

Write
\begin{align*}
  \Wodd(k)&=\W(1,\veq1,k+1,k+2,2k+2),\\
  \Weven(k)&=\W(1,\veq1,k+2,k+2,2k+3),
\end{align*}
so initially, $\a\in\W(1,\veq1,2,3,4) = \Wodd(1)$.  The exit
conditions are:
\begin{itemize}
  \item for $\a\in\Wodd(k)$: that $y^2+a_{3,k+1}y-a_{6,2k+2}$ has distinct roots
    over~$\F_p$, or equivalently that
    $\ordp(b_6)=2k+2$, and for $p=2$ to $\ordp(a_3)=k+1$;
  \item for $\a\in\Weven(k)$: that $x^2+a_{4,k+2}x+a_{6,2k+3}$ have distinct
    roots over~$\F_p$, equivalently that $\ordp(b_8)=2k+4$, or that
    $\ordp(a_4)=k+2$ when $p=2$.
\end{itemize}

First assume that $p\not=2$.  For $\a\in\Wodd(k)$ we have
\(
\Delta \equiv -2^4a_2^3b_6 \pmod{p^{2k+6}};
\)
since $\ordp(a_2)=1$, when the exit condition $v(b_6)=2k+2$ holds,
for $p\not=2$ we have $n=2k+5$ exactly.  Hence $n=m+6$ and $f_p=2$.
Otherwise, after shifting the double root to~$0$ by a suitable
translation, we arrive in~$\Weven(k)$, where
\(
\Delta \equiv -2^4b_8a_2^2 \pmod{p^{2k+7}};
\)
when the exit condition $\ordp(b_8)=2k+4$ holds, we have
$v(2^4b_8a_2^2) = 2k+6$, so $n=2k+6$. Again, $n=m+6$ and~$f_p=2$.
Otherwise, after another shift we arrive in $\Wodd(k+1)$, so we
increment $k$ and repeat.

Hence for $p\ge3$, we always have $f_p=2$.

Now let $p=2$.  Again, the value of~$m$ is initialized to~$1$ and we
proceed recursively; at each stage we either exit (always with
relative probability~$1/2$), or increment~$m$.  The recursive steps
alternate in nature depending on the parity of~$m$; after the first
three cases ($m=1, 2, 3$) which are slightly different, all the
remaining cases may be dealt with generically.

At first, $m=1$ with $\a\in\Wodd(1)=\W(1,\veq1,2,3,4)$, where we have
\(
\Delta \equiv a_3^4 \pmod{2^9}.
\)
When the exit condition $\ordp(a_3)=2$ holds, we have $n=8$ and
$f_2=3$, and Type~$\I_1^*$.

Otherwise, $\ordp(a_3)\ge3$, and we shift~$y$ so that the quadratic
$y^2+a_{3,2}^2y-a_{6,4}$ has double root at $y\equiv0\pmod{2}$, so
that~$\a\in\Weven(1)=\W(1,\veq1,3,3,5)$, and we increment~$m$ to~$2$.

Now we have
\(
\Delta \equiv a_1^4a_4^2\pmod{2^{11}},
\)
so $n\ge10$.  When the exit condition $\ordp(a_4)=3$ holds, either
$\ordp(a_1)=1$, giving $n=10$ and $f_2=4$; or $\ordp(a_1)\ge2$, and
$\a\in\W(2,\veq1,3,\veq3,5)$.  In the latter case,
\(
\Delta \equiv a_3^4 + 2^{12} \pmod{2^{14}},
\)
so $n=12$ if $\ordp(a_3)\ge4$, and $n=13$ if $\ordp(a_3)=3$. Hence for
Type~$\I_2^*$ we have $n=10, 12$, or~$13$ (respectively, $f_2=4, 6$,
or~$7$) with relative probabilities~$1/2,1/4$, and~$1/4$.

Otherwise, when the exit condition at~$m=2$ fails, we have
$\ordp(a_4)\ge4$, and we shift~$x$ so that the quadratic
$x^2+a_{4,3}x+a_{6,5}$ has double root at $x\equiv0\pmod{2}$, so that
$\a\in\Wodd(2) = \W(1,\veq1,3,4,6)$ and we increment~$m$ to~$3$.

Now we have
\(
\Delta\equiv2 a_1^4a_3^2\pmod{2^{12}},
\)
so $n\ge11$.  When the exit condition $\ordp(a_3)=3$ holds, either
$\ordp(a_1)=1$, giving $n=11$ and $f_2=4$; or $\ordp(a_1)\ge2$, and
$\a\in\W(2,\veq1,\veq3,4,6)$. In the latter case,
\(
\Delta\equiv2^{12}\pmod{2^{13}},
\)
so $n=12$ and $f_2=5$.  Hence for Type~$\I_3^*$, we have $n=11$ or~$12$
(respectively, $f_2=4$ or~$5$) with equal probability.

Now let $m=2k\ge4$, with $\a\in\Weven(k)$, and exit condition
$\ordp(a_4)=k+2$.  Then
\(
\Delta \equiv a_1^4a_4^2 \pmod{2^{2k+9}},
\)
so $n\ge \ordp(a_1^4a_4^2)=2k+8$. Assuming that the exit condition $\ordp(a_4)=k+2$
holds, we have $n=2k+8=m+8$ and $f_2=4$,
provided that $\ordp(a_1)=1$.  Otherwise, $\ordp(a_1)\ge2$, and now
\(
\Delta\equiv 2^{2k+10}\pmod{2^{2k+11}},
\)
so $n=2k+10=m+10$ and $f_2=6$.  Thus for $m=2k\ge4$ we have $f_2=4$
or~$f_2=6$, with equal probability.

If the exit condition fails, $\ordp(a_4)\ge k+3$, and we may shift~$x$
so that the quadratic $x^2+a_{4,k+2}x+a_{6,2k+3}$ has its double root at
$x\equiv0\pmod{2}$, so also $\ordp(a_6)\ge 2k+4$ and
$\a\in\Wodd(k+1)$.  Incrementing both~$k$ and~$m$ so that $m=2k-1$, we
have~$\a\in\Wodd(k)$.

Next, $m=2k-1\ge5$, with $\a\in\Wodd(k)$ and exit condition
$\ordp(a_3)=k+1$.  Now,
\(
\Delta\equiv 2 a_1^4a_3^2 \pmod{2^{2k+8}},
\)
so $n\ge \ordp(2a_1^4a_3^2) = 2k+7$.  Assuming that the exit condition $\ordp(a_3)=k+1$
holds, we have $n=2k+7=m+8$ and $f_2=4$,
provided that $\ordp(a_1)=1$.  Otherwise, $\ordp(a_1)\ge2$, and now
\(
\Delta\equiv 2^{2k+9}\pmod{2^{2k+10}},
\)
so $n=2k+9=m+10$ and $f_2=6$. Thus for~$m=2k-1\ge5$, we again have
$f_2=4$ or~$f_2=6$ with equal probability.

If the exit condition fails, $\ordp(a_3)\ge k+2$, and we may shift~$y$
so that the quadratic $y^2+a_{3,k+1}y-a_{6,2k+2}$ has double root at $y\equiv0$,
so that~$\a\in\Weven(k)$, and we increment~$m$ to~$2k$ and recurse.

Taking all Types~$\I_m^*$ for $m\ge1$ together, we find that
$f_2=3,4,5,6$, or~$7$ with relative probabilities $1/2,1/4,1/16,1/8$,
and~$1/16$.

This completes the analysis of type $\I_m^*$.

\subsubsection*{Additive reduction, Type~$\IV^*$}
The exit condition for type~$\I_m^*$ fails when the cubic $g(x)$
has a triple root; after the transform moving the root to~$0$, this
means that $\ordp(a_2)\ge2$, so $\a\in\W(1,2,2,3,4)$, with the same
stabiliser as for $\W(1,\veq1,2,3,4)$, namely $\TT212$.

The exit condition for type~$\IV^*$ is that $f=y^2+a_{3,2}y-a_{6,4}$ has
distinct roots modulo~$p$, or equivalently $\ordp(b_6)=4$, which
happens with probability~$1-1/p$.  Thus
\begin{equation}\label{eqn:B_IVs}
\mu(\W(1,2,2,3,4\mid\ordp(b_6)=4)) = \frac{p-1}{p}\mu(\W(1,2,2,3,4)) =
(p-1)/p^{13}.
\end{equation}
Now, $c_p=1$ or $c_p=3$, depending on whether the roots are in~$\F_p$
or not, and these have equal probability by
Lemma~\ref{lem:quad-count}.

To compute $f_p=n-6$, we first note that for $\a\in\W(1,2,2,3,4)$ we
have
\(
\Delta \equiv -3^3b_6^2 \pmod{3^9};
\)
when $\ordp(b_6)=4$, this implies that $n=8$ and
$f_p=2$ provided that $p\not=3$.

Now consider $p=3$.  We have $\ordp(b_2)\ge2$, $\ordp(b_4)\ge3$, and
\(
\Delta \equiv b_4^3 \pmod{3^{10}},
\)
so $n\ge9$, and $n=9\iff \ordp(b_4)=3$, which is equivalent to
$a_{4,3}\not\equiv a_{1,1}a_{3,2}\pmod3$, so has relative
probability~$2/3$. Assuming that $a_{4,3}\equiv a_{1,1}a_{3,2}\pmod{3}$, we find that
\(
\Delta\equiv-3^4b_2b_6\pmod{3^{10}}.
\)
Hence $n\ge10$, with $n=10\iff v(b_2)=2 \iff a_{1,1}^2+a_{2,2}\not\equiv0\pmod3$.
Assuming further that $a_2'\equiv-a_1'^2\pmod3$,  we have
\(
\Delta \equiv -3^3b_6^2\pmod{3^{12}},
\)
so $n=11$ exactly.  Hence for $p=3$ we have $n=9, 10$, or~$11$
(respectively, $f_3=3,4$, or~$5$) with relative
probabilities~$2/3,2/9$, and~$1/9$.

\subsubsection*{Additive reduction, Type~$\III^*$}
When the exit condition for type~$\IV^*$ fails, we move the root of
the quadratic to~$0$ using a transform in a unique coset of~$\TT213$
in $\TT212$ to arrive in~$\W(1,2,3,3,5)$ with stabiliser~$\TT213$.

The exit condition for type~$\III^*$ is $\ordp(a_4)=3$.  In all cases
we have $c_p=2$.  To compute $f_p=n-7$, we first note that for
$\a\in\W(1,2,3,3,5)$ we have
\(
\Delta\equiv-2^6a_4^3\pmod{p^{10}};
\)
when $\ordp(a_4)=3$, this implies that $n=9$ and~$f_p=2$, provided
that $p\not=2$.

Let $p=2$. For $\a\in\W(1,2,3,3,5)$ we now have
\(
\Delta\equiv a_1^4a_4^2 \pmod{2^{11}},
\)
so $n\ge 10$, and when the exit condition $\ordp(a_4)=3$ holds,
we have $n=10 \iff \ordp(a_1)=1$.
Assuming that $\ordp(a_1)\ge2$, so $\a\in\W(2,2,3,3,5)$, we have
\(
\Delta\equiv a_3^4 \pmod{2^{13}},
\)
so $n\ge12$, with $n=12 \iff \ordp(a_3)=3$.
Assuming further that $\ordp(a_3)\ge4$, so $\a\in\W(2,2,4,3,5)$, we have
\(
\Delta \equiv 2^{4}(2^2a_1^4+2^6a_2^2+a_6^2) \equiv 2^{14}(a_{1,2}^4+a_{2,2}^2+a_{6,5}^2) \pmod{2^{15}},
\)
so $n\ge14$, with $n=14\iff a_{6,5}\not\equiv a_{1,2}+a_{2,2}\pmod2$.
Assuming that $a_{6,5}\equiv a_{1,2}+a_{2,2}\pmod2$, we find that
\(
\Delta \equiv 2^{15}\pmod{2^{16}},
\)
so that $n=15$.

Hence for $p=2$ we have $n=10, 12, 14$, or~$15$ (respectively, $f_2=3,
5, 7$, or~$8$) with relative probability~$1/2, 1/4, 1/8$ and~$1/8$.

\subsubsection*{Additive reduction, Type~$\II^*$}
When the exit condition for type~$\III^*$ fails we are
in~$\W(1,2,3,4,5)$ with the same stabiliser~$\TT213$, since
$\TT213$ preserves the condition $\ordp(a_3)=3$.

The exit condition for type~$\II^*$ is $\ordp(a_6)=5$.  In all cases
we have $c_p=1$, and $f_p=n-8$.

For $\a\in\W(1,2,3,4,5)$, we have
\(
\Delta\equiv-2^43^3 a_6^2 \pmod{p^{11}},
\)
so when the exit condition holds we have $n=10$ and $f=2$ for
all~$p\ge5$.

Let $p=3$.  Now, $\ordp(b_2)\ge2$, and
\(
\Delta\equiv -a_6b_2^3 \pmod{3^{12}};
\)
hence $n\ge11$, with $n=11\iff \ordp(b_2)=2 \iff a_{2,2}\not\equiv -a_{1,1}^2\pmod3$.
Assuming that $a_{2,2}\equiv-a_{1,1}^2\pmod3$, we find that
\(
\Delta\equiv b_4^3\pmod{3^{13}},
\)
so $n\ge12$, with $n=12\iff v(b_4)=3 \iff a_{4,4}\not\equiv
a_{1,1}a_{3,3}\pmod3$.
Assuming further that $a_{4,4}\equiv a_{1,1}a_{3,3}\pmod3$, we find that
\(
\Delta\equiv-3^3a_6\pmod{3^{14}},
\)
so $n=13$.  Hence for $p=3$ we have $n=11, 12$, or~$13$ (respectively,
$f_3=3, 4$, or~$5$) with relative probabilities~$2/3,2/9$, and~$1/9$.

Finally, let $p=2$. Now
\(
\Delta\equiv a_1^6a_6 \pmod{2^{12}},
\)
so $n\ge11$, with $n=11 \iff \ordp(a_1)=1$. If
$\ordp(a_1)\ge2$, then
\(
\Delta \equiv a_3^4 \pmod{2^{13}},
\)
so $n\ge12$, with $n=12 \iff\ordp(a_3)=3$.  If also
$\ordp(a_3)\ge4$, then
\(
\Delta \equiv 2^{14} \pmod{2^{15}},
\)
so $n=14$. Hence for $p=2$ we have $n=11, 12$, or~$14$ (respectively,
$f_2=3, 4$, or~$6$) with relative probability~$1/2, 1/4$ and~$1/4$.

When the exit condition for type~$\II^*$ fails we are
in~$\W(1,2,3,4,6)$ with the same stabiliser~$\TT213$.

This completes the proof of Theorem~\ref{thm:fp-cp-count}.
\qed

\subsection{Distribution of conductor exponents}

Finally, we collect together the possible conductor exponents over all
reduction types, to find the overall density of each.  Here we omit
non-minimal models, so the densities add up to $1-1/p^{10}$.

\begin{thm}[Overall distribution of conductor exponents]\label{thm:fp-all}
The overall densities of conductor exponents~$f_p$ for minimal
Weierstrass models over $\Z_p$ are as follows:
\begin{enumerate}
\item Good and multiplicative reduction
\[
\begin{tabular}{cc}
\hline
$f_p$ & density \\
\hline
$0$ & $1 - 1/p$ \\
$1$ & $1/p - 1/p^2$ \\
\hline
\end{tabular}
\]

\item Additive reduction.

\begin{itemize}
\item[$p\ge5$.]
\[
\begin{tabular}{cc}
\hline
$f_p$ & density \\
\hline
$2$ & $1/p^2 - 1/p^{10}$ \\
\hline
\end{tabular}
\]

\item[$p=3$.] The following densities add up to $59040/3^{12}=1/3^2-1/3^{10}$:
\[
\begin{tabular}{cr}
\hline
$f_p$ & density \\
\hline
$2$ & $15120/3^{12}$ \\
$3$ & $29280/3^{12}$ \\
$4$ & $9760/3^{12}$ \\
$5$ & $4880/3^{12}$ \\
\hline
\end{tabular}
\]

\item[$p=2$.] The following densities add up to $1020/2^{12}=1/2^2-1/2^{10}$:
\[
\begin{tabular}{cr}
\hline
$f_p$ & density \\
\hline
$2$ & $144/2^{12}$ \\
$3$ & $150/2^{12}$ \\
$4$ & $297/2^{12}$ \\
$5$ & $84/2^{12}$ \\
$6$ & $213/2^{12}$ \\
$7$ & $99/2^{12}$ \\
$8$ & $33/2^{12}$ \\
\hline
\end{tabular}
\]
\end{itemize}
\end{enumerate}
\end{thm}
\begin{proof}
Immediate from ~\ref{thm:fp-cp-count}.
\end{proof}

\bibliographystyle{amsplain}
\providecommand{\bysame}{\leavevmode\hbox to3em{\hrulefill}\thinspace}
\providecommand{\MR}{\relax\ifhmode\unskip\space\fi MR }
% \MRhref is called by the amsart/book/proc definition of \MR.
\providecommand{\MRhref}[2]{%
  \href{http://www.ams.org/mathscinet-getitem?mr=#1}{#2}
}
\providecommand{\href}[2]{#2}

%\bibliography{refs}

\end{document}